\documentclass[a4paper]{paper}

\usepackage{paper}

\title{Nonabelian Landau-Ginzburg orbifolds and Calabi-Yau/Landau-Ginzburg correspondence}
\author{Daichi Mukai}
\address{
  Research Institute for Mathematical Sciences, Kyoto University, Kyoto 606-8502, Japan
}
\email{daichi@kurims.kyoto-u.ac.jp}

\begin{document}

\begin{abstract}
  In this paper, we study the bigraded vector space structure of Landau-Ginzburg orbifolds.
  We prove the formula for the generating function of the Hodge numbers of possibly nonabelian
  Landau-Ginzburg orbifolds.
  As an application, we calculate the Hodge numbers for all nondegenerate quintic homogeneous
  polynomials with five variables.
  These results yield an evidence for the Calabi-Yau/Landau-Ginzburg correspondence between
  the Calabi-Yau geometries and the Landau-Ginzburg B-models.
\end{abstract}

\maketitle

\section{Introduction}
  \label{sec:introduction}
  The $\mathcal{N} = 2$ superconformal Landau-Ginzburg model is a $2$-dimensional conformal field theory
  determined by a nondegenerate quasihomogeneous polynomial $W$.
  This model possesses the symmetry represented by the $\mathcal{N} = 2$ superconformal algebra generated by
  a basis $L_{n}$, $J_{n}$, $G_{r}^{\pm}$ and $c$ where $n \in \Z$ and $r$ runs over half-integral or integral
  values.
  The formar is called the Neveu-Schwarz sector and the latter is called the Ramond sector.
  Both of the two sectors satisfy the same canonical (anti)commutation relations.
  Although the difference between the definitions of the two sectors is a set of indices of $G_{r}^{\pm}$,
  the structures of the representation of these sectors are considerably different.

  A primary chiral state is the state $\ket{\phi}$ in the Neveu-Schwarz state space of an $\mathcal{N} = 2$
  superconformal field theory which satisfies the condition
  \begin{equation*}
    G_{n-1/2}^{+} \ket{\phi} = G_{n+1/2}^{-} \ket{\phi} = 0 \quad \text{for}\ n \geq 0.
  \end{equation*}
  Using the anticommutation relation, we deduce for primary chiral states
  \begin{equation*}
    \{G_{-1/2}^{+}, G_{1/2}^{-}\} \ket{\phi} = (2L_{0} - J_{0}) \ket{\phi} = 0.
  \end{equation*}
  Therefore the eigenvalue of $L_{0}$ (the conformal dimension $h$) is half of the eigenvalue of $J_{0}$
  (the charge $q$), i.e. $h = q/2$.

  The chiral ring is the operator algebra of the fields corresponding to the primary chiral states.
  The multiplication of this algebra is given by the usual operator product of chiral fields modulo
  setting the descendant chiral fields to zero.
  For example, the chiral ring associated to a Landau-Ginzburg model is isomorphic to the Jacobi ring of $W$.

  The Landau-Ginzburg orbifold is a $\mathcal{N} = 2$ superconformal field theory obtained by orbifolding
  a Landau-Ginzburg model by a symmetry group $G$ of $W$.
  The subject of the paper is a chiral ring of a Landau-Ginzburg orbifold which we call simply
  a Landau-Ginzburg orbifold.
  This orbifold theory was given by Intriligator and Vafa~\cite{intriligator1990}
  as the vector space which is essentially a direct sum of some Jacobi rings associated to
  the nondegenerate quasihomogeneous polynomials determined by $W$ and $G$.
  However, this definition was lacking in a product structure.
  Recently, Krawitz~\cite{krawitz2010} showed, besides writing down a product structure, that the
  Landau-Ginzburg orbifold has the structure of a Frobenius algebra under some assumptions
  if $W$ is a so-called invertible polynomial and $G$ is a so-called diagonal symmetry group.
  The Frobenius pairing physically corresponds to the correlation function of the conformal field theory and
  hence plays an important role.

  Such a Landau-Ginzburg orbifold constructed from Jacobi rings is nowadays called
  the ``Landau-Ginzburg B-model''.
  Another Landau-Ginburg model, called ``Landau-Ginzburg A-model'', had arisen from the study of
  topological gravity.
  The A-model is also an orbifold theory associated to a nondegenerate quasihomogeneous polynomial
  and its symmetry group.
  Its Frobenius algebra structure was recently studied by Fan, Jarvis, and Ruan~\cite{fan2013}.

  The Landau-Ginzburg mirror symmetry conjecture states that the equality between the A-model
  of a given polynomial $W$ and the symmetry group $G$, and the B-model of an appropriately chosen polynomial
  $\check{W}$ and group $\check{G}$.
  For example, it has been proved on the level of Frobenius algebras when $W$ is an invertible polynomial
  and $G$ is diagonal symmetry group satisfying a certain property~\cite{krawitz2010, francis2012}.
  In addition, if $G$ is a maximal diagonal symmetry group, it can be generalized for any genus~\cite{he2015}.
  However, this conjecture is an open problem for general cases.

  There is another duality which relates the Landau-Ginzburg model and the Calabi-Yau geometry,
  called the Calabi-Yau/Landau-Ginzburg (CY/LG) correspondence.
  Using the Fan-Jarvis-Ruan-Witten (FJRW) theory as the candidate theory for the Landau-Ginzburg A-model,
  Chiodo and Ruan~\cite{chiodo2011} tried to understand this correspondence between the orbifold
  cohomology of the orbifold $X_{W} / \Tilde{G}$ defined by Chen and Ruan~\cite{chen2004} and the FJRW theory.
  Here, $X_{W}$ is a set of solutions of $W$ in the weighted projective space and $\Tilde{G}$ is the image
  of $G$ in $\Aut{(X_{W})}$.
  This attempt has succeeded for the case where $W$ is a nondegenerate quasihomogeneous polynomial
  satisfying the Calabi-Yau condition (see Definition~\ref{def:calabi-yau-condition}) and $G$ is a
  diagonal symmetry group of $W$.

  Let us summarize these conjectural dualities related to Landau-Ginzburg models and its geometric
  counterparts.

  
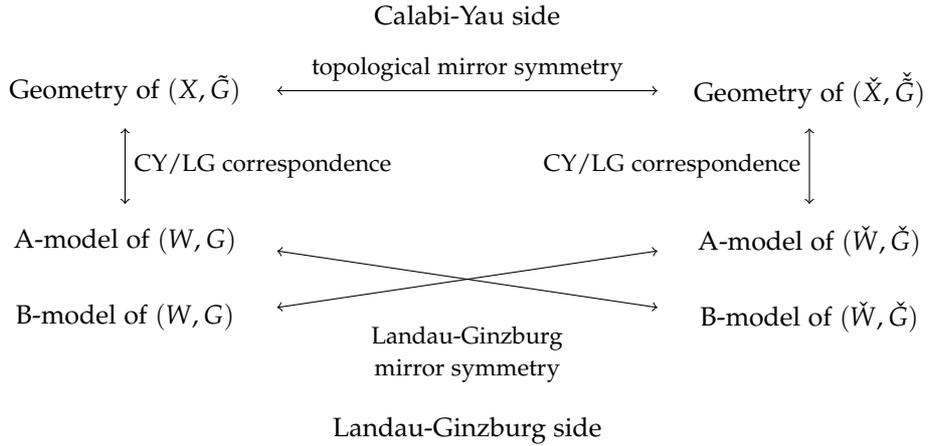
\begin{figure}[h!]
    \centering
    \begin{tikzpicture}
      \node at (4.5,3) {Calabi-Yau side};
      \node at (4.5,-2.5) {Landau-Ginzburg side};
      \node at (0,2) {Geometry of $(X, \Tilde{G})$};
      \draw [<->] (0,1.5) -- (0,0.5) node [midway,right] {\small CY/LG correspondence};
      \node at (0,0) {A-model of $(W, G)$};
      \node at (0,-1) {B-model of $(W, G)$};
      \node at (9,2) {Geometry of $(\check{X}, \check{\Tilde{G}})$};
      \draw [<->] (9,1.5) -- (9,0.5) node [midway,left] {\small CY/LG correspondence};
      \node at (9,0) {A-model of $(\check{W}, \check{G})$};
      \node at (9,-1) {B-model of $(\check{W}, \check{G})$};
      \draw [<->] (2,-0.125) -- (7,-0.875);
      \draw [<->] (2,-0.875) -- (7,-0.125);
      \draw [<->] (2,2) -- (7,2) node [midway,above,align=center] {\small topological mirror symmetry};
      \node [align=center,below] at (4.5, -1) {\small Landau-Ginzburg \\ \small mirror symmetry};
    \end{tikzpicture}
  \end{figure}

  Although the Landau-Ginzburg model could be defined with an arbitrary finite symmetry group,
  most of the precedent researches were focused on the diagonal ones.
  The main theme of the paper is the vector space structure of the Landau-Ginzburg B-model
  with a possibly nonabelian symmetry group.
  This vector space is equipped with the bigrading which reflects the left-right charge of the corresponding
  conformal field theory.
  Defining the Hodge number as the dimension of the homogeneous subspace of this bigrading,
  we see that this Hodge numbers satisfy similar relations with the ordinary Hodge numbers of a compact
  Kähler manifold such as Hodge symmetry and Serre duality.
  Note that an all genus quantum theory of Landau-Ginzburg A-model with a possibly nonabelian group has been
  obtained as a special case of gauged linear sigma models~\cite{fan2015}.

  Recently Oguiso and Yu~\cite{oguiso2015} completely classified automorphism groups of nonsingular quintic
  threefolds.
  Using this result, Yu~\cite{Yu_2016} calculated the orbifold Hodge numbers of quintic threefold orbifolds
  for all automorphism groups which fix a nowhere vanishing holomorphic $3$-form on the manifold.
  This list of the groups by Yu give us a partial classification of the symmetry group of nondegenerate
  quintic polynomial.
  Computing the Hodge numbers using the groups in this list, we obtain the main result
  Theorem~\ref{thm:main-theorem} in the paper.
  This gives credence to the CY/LG correspondence for the nondegenerate quintic polynomial case.
  Note that this CY/LG correspondence is another side of CY/LG correspondence proved
  by Chiodo and Ruan since we discuss the Calabi-Yau geometry and the Landau-Ginzburg B-model, not A-model.

  \subsection*{Organization of the paper}
    The paper is organized as follows.
    In Section~\ref{sec:nondegenerate-quasihomogeneous-polynomial} and Section~\ref{sec:symmetry-group}, we
    review the definition of a nondegenerate quasihomogeneous polynomial and its symmetry group.
    These are background data for the Landau-Ginzburg orbifold.
    In Section~\ref{sec:landau-ginzburg-orbifold}, we define the Landau-Ginzburg orbifold as a bigraded vector
    space following Intriligator and Vafa~\cite{intriligator1990}.
    In Section~\ref{sec:generating-function}, the formula for the generating function of
    the Hodge numbers is given and some basic properties of the Hodge number are proved using this formula.
    Section~\ref{sec:orbifold-quintic-threefold} is devoted to a review on the result of Yu~\cite{Yu_2016}.
    In Section~\ref{sec:quintic-landau-ginzburg-orbifold}, we calculate the Hodge numbers of various Landau-Ginzburg orbifolds
    and prove our main therem in the paper.
    In Section~\ref{sec:conclusion}, we conclude the paper with commenting on some future prospects related to the Landau-Ginzburg orbifold.
    Appendix~\ref{sec:table} contains two tables which are part of Theorem~\ref{thm:nonsingular-quintic-threefold-hodge-number}
    and Theorem~\ref{thm:main-theorem}.
    Appendix~\ref{sec:hodge-number-calculation} describes a technical remark on the computation of
    the Hodge numbers for the proof of the main theorem.

  \subsection*{Acknowledgment}
    The author would like to thank his supervisor Toshiya Kawai for many discussions and valuable advice which improved
    this paper.

\section{Nondegenerate quasihomogeneous polynomial}
  \label{sec:nondegenerate-quasihomogeneous-polynomial}
  In this section, we define a quasihomogeneous polynomial and its nondegeneracy.
  A nondegenerate quasihomogeneous polynomial is one of background data of Landau-Ginzburg orbifold.

  \begin{definition}
    Let $W$ be a polynomial in $\C[x_{1},\dots,x_{n}]$ with a critical point at the origin.
    The ideal
    \begin{equation*}
      I_{\nabla W} = \left(\frac{\partial W}{\partial x_{1}},\dots,\frac{\partial W}{\partial x_{n}}\right)
      \subset \C[x_{1},\dots,x_{n}]
    \end{equation*}
    generated by the partial derivatives of $W$ is called the \emph{gradient ideal} of $W$.
    The quotient
    \begin{equation*}
      \mathcal{Q}_{W} = \C[x_{1},\dots,x_{n}]/I_{\nabla W}
    \end{equation*}
    by the gradient ideal is called the \emph{Jacobi ring} of $W$.
    The critical point of $W$ at the origin is said to be \emph{isolated} if the Jacobi ring of $W$ is
    finite dimensional as a vector space over $\C$.
  \end{definition}

  \begin{remark}
    \label{rem:jacobi-ring-invariant-singularity}
    The Jacobi ring $\mathcal{Q}_{W}$ is invariant under a linear change of the coordinates.
    More precisely, a linear change of the coordinates induces a transition of the exact sequence of $\C$-algebra
    \begin{equation*}
      0 \to I_{\nabla W} \to \C[x_{1}, \dots, x_{n}] \to \mathcal{Q}_{W} \to 0
    \end{equation*}
    to an isomorphic exact sequence~\cite{arnold1998}.
  \end{remark}

  \begin{definition}
    Let $W$ be a polynomial in $\C[x_{1},\dots,x_{n}]$.
    $W$ is called \emph{quasihomogeneous} if there exist positive integers $w_{1},\dots,w_{n}, d$
    which satisfy
    \begin{equation*}
      W(\lambda^{w_{1}} x_{1},\dots,\lambda^{w_{n}} x_{n}) = \lambda^{d} W(x_{1},\dots,x_{n})
    \end{equation*}
    for all $\lambda \in \C$.
    $w_i$ and $q_{i} \defeq w_{i}/d$ are called the \emph{weight} and the \emph{charge} of $x_{i}$ respectively.
    We denote the degree of a polynomial $f(x)$ which is weighted by charges $q_i$ by $\text{charge}(f)$.
  \end{definition}

  \begin{definition}
    \label{def:nondegenerate}
    A quasihomogeneous polynomial $W$ is said to be \emph{nondegenerate} if $W$ has an isolated critical point at the origin.
  \end{definition}

  \begin{example}
    \label{ex:type-An}
    The quasihomogeneous polynomials with one variable is homogeneous.
    More precisely, $W(x)$ in $\C[x]$ is quasihomogeneous if and only if $W(x)$ is the form of $c x^{n+1}$ for $c \in \Cx$ and $n \geq 1$.
    Then, the charge of $x$ is $1/n+1$.
    This type of singularity is called type $A_{n}$.
  \end{example}

  \begin{example}
    A quasihomogeneous polynomial is said to be \emph{invertible} if the number of monomials equals the number of variables.
    The invertible quasihomogeneous polynomial can be rescaled so that each monomial has a coefficient $1$.
    The invertible quasihomogeneous polynomials are completely classified by Kreuzer and Skarke~\cite{kreuzer1992}.
  \end{example}

  \begin{theorem}[Kreuzer-Skarke~\cite{kreuzer1992}]
    Any invertible quasihomogeneous polynomial is a decoupled sum of polynomials of one of the following two atomic types:
    \begin{description}
      \item[Loop type] $x_{1}^{a_{1}}x_{2} + x_{2}^{a_{2}}x_{3} + \dots + x_{n}^{a_{n}}x_{1}$ ($n \geq 2$),
      \item[Chain type] $x_{1}^{a_{1}}x_{2} + x_{2}^{a_{2}}x_{3} + \dots + x_{n}^{a_{n}}$ ($n \geq 1$).
    \end{description}
  \end{theorem}

    The atomic type with $n = 1$ for the chain type, i.e. the term of the form $x^{a}$, is sometimes called \emph{Fermat} type.

  \begin{lemma}
    \label{lem:findim-iff-isolated}
    Let $W(x_1, \dots, x_n)$ be a quasihomogeneous polynomial.
    $W$ is nondegenerate if and only if
    $\frac{\partial W}{\partial x_1} = \dots = \frac{\partial W}{\partial x_n} = 0$
    implies $x_1 = \dots = x_n = 0$.
  \end{lemma}

  \begin{proof}
    Assume that $W$ is nondegenerate, i.e. $\mathcal{Q}_W \defeq \C[x_1, \dots, x_n]/I_{\nabla W}$ is finite dimensional.
    Then $\mathcal{Q}_W$ is Artinian, and therefore $\Spec{\mathcal{Q}_W}$ is finite.
    In particular, $\mathcal{Q}_W$ has finitely many maximal ideals.
    Clearly, $(x_1, \dots, x_n)/I_{\nabla W}$ is a maximal ideal.
    If $(x_1 - a_1, \dots, x_n - a_n)/I_{\nabla W}$ is maximal for $(a_1, \dots, a_n) \neq (0, \dots, 0)$,
    $(x_1 - k a_1, \dots, x_n - k a_n)/I_{\nabla W}$ is also maximal for any $k \in \Cx$, yielding
    the contradiction.
    Conversely, assume that $\frac{\partial W}{\partial x_1} = \dots = \frac{\partial W}{\partial x_n} = 0$ implies $x_1 = \dots = x_n = 0$.
    Then we see $\Radical{I_{\nabla W}} = (x_1, \dots, x_n)$, and hence
    $(x_1, \dots, x_n)^k \subset I_{\nabla W}$ for sufficiently large $k$.
    $\C[x_1, \dots, x_n]/(x_1, \dots, x_n)^k$ has only one prime ideal $(x_1, \dots, x_n)/(x_1, \dots, x_n)^k$ and therefore Artinian.
    Since Artinian $\C$-algebra is finite dimensional, we have
    \begin{equation*}
      \dim_{\C} \mathcal{Q}_W \leq \dim_{\C} \C[x_1, \dots, x_n]/(x_1, \dots, x_n)^k < \infty. \qedhere
    \end{equation*}
  \end{proof}

  \begin{lemma}
    \label{lem:nondegenerate-contain-term}
    Let $W(x_1, \dots, x_n)$ be a nondegenerate quasihomogeneous polynomial.
    Then for each $i$, $W$ contains a monomial of the form $x_i^a x_j$ for some $j$ and $a$.
  \end{lemma}

  \begin{proof}
    If $n = 1$, this statement is trivial because $W$ must be $c x_{1}^{a}$ where $c \in \Cx$ and $a \geq 2$.
    For $n \geq 2$, assume otherwise for $i=1$.
    Then we have $W \in (x_2, \dots, x_n)^2$.
    Hence $(a_1, 0, \dots, 0)$ for any $a_1 \in \C$ is a solution of
    $\frac{\partial W}{\partial x_1} = \dots = \frac{\partial W}{\partial x_n} = 0$
    which yields the contradiction by Lemma~\ref{lem:findim-iff-isolated}.
  \end{proof}

  \begin{remark}
    Lemma~\ref{lem:nondegenerate-contain-term} does not give a sufficient condition.
    For example, let $W(x_{1}, x_{2}, x_{3}) = x_{1}^{3} + x_{2}^{3} + x_{3}^{3} - 3 \psi x_{1} x_{2} x_{3}$ where $\psi \in \C$.
    Then $W$ is a quasihomogeneous polynomial and contains the terms $x_{i}^{3}$ for $i = 0, 1, 2$.
    However, $W$ is nondegenerate if and only if $\psi^{3} \neq 1$.
  \end{remark}

  For a nondegenerate quasihomogeneous polynomial with weights $w_{1}$, $\dots$, $w_{n}$, we can assume $\gcd(w_1, \dots, w_n) = 1$.
  If a quasihomogeneous polynomial does not contain a term of the form which is proportional to $x_{i}x_{j}$, the weights
  are uniquely determined by Lemma~\ref{lem:nondegenerate-contain-term} and furthermore charges satisfy $q_{i} \geq 1/2$.

\section{Symmetry group of nondegenerate quasihomogeneous polynomial}
  \label{sec:symmetry-group}
  In this section, we define the symmetry group of a quasihomogeneous polynomial.
  The Landau-Ginzburg orbifolds are constructed from the pair of a nondegenerate quasihomogeneous polynomial and
  a finite subgroup of the symmetry group of the polynomial.

  Let $V$ be a $n$-dimensional $\C$-vector space with a basis $x_1, \dots, x_n$.
  We define the action of $\GL_{n}(\C)$ on $V$ by $g{\cdot}x_i = \sum_{j=1}^{n}g_{i j}x_{j}$ for any $g = (g_{i j})$ in $\GL_{n}(\C)$ and extending linearly.
  This action induces the action of $\GL_{n}(\C)$ on the polynomial ring $\C[x_{1}, \dots, x_{n}]$ by $(g{\cdot}f)(x_{1}, \dots, x_{n}) = f(g{\cdot}x_{1}, \dots, g{\cdot}x_{n})$.

  \begin{definition}
    Let $W(x_{1},\dots,x_{n})$ be a nondegenerate quasihomogeneous polynomial
    with weights $w_1, \dots, w_n$.
    The group
    \begin{equation*}
      G_{W} = \Set{g \in \GL_n(\C) |
        \begin{array}{c}
          (g{\cdot}W)(x_1, \dots, x_n) = W(x_1, \dots, x_n) \\
          g_{i j} = 0 \ \text{if} \ w_i \neq w_j
        \end{array}
      }
    \end{equation*}
    is called the \emph{(maximal) symmetry group} of $W$.
  \end{definition}

  Exchanging the indices of the coordinates, we can assume $w_1 \leq \dots \leq w_n$.
  Then the condition $g_{i j} = 0$ if $w_i \neq w_j$ means that $g$ has a block diagonal form.
  Hence, this condition is equivalent to the condition that each $g$ commutes with the action of $\lambda \in \Cx$ where
  $\lambda$ acts on $(x_{1}, \dots, x_{n})$ by $\lambda{\cdot}(x_{1}, \dots, x_{n}) = (\lambda^{w_{1}}x_{1}, \dots, \lambda^{w_{n}}x_{n})$.
  \footnote{The author thanks Y. Ruan for pointing out this equivalence.}

  The symmetry group $G_{W}$ of any nondegenerate quasihomogeneous polynomial $W(x_1, \dots, x_n)$ is not trivial
  since it contains
  \begin{equation*}
    J \defeq \diag\left(\e{q_1},\dots,\e{q_n}\right)
  \end{equation*}
  which is called the \emph{exponential grading operator} in some literature.
  Note that the exponential grading operator is an element of the center of $G_W$.
  In the following, we discuss on orbifolding by a finite subgroup of $G_W$.

  \begin{remark}
    A subgroup $G$ of the maximal symmetry group is called a \emph{diagonal symmetry group} if $G$ consists of diagonal
    matrices.
    Clearly, a diagonal symmetry group is abelian.
    In many precedent researches on the Landau-Ginzburg A- or B-models, orbifolds are discussed
    assuming that the symmetry groups are diagonal ones.
    One of the main theme of the paper is to consider Landau-Ginzburg orbifolds constructed from possibly
    nonabelian symmetry groups besides diagonal ones.
  \end{remark}

  \begin{definition}
    Let $G$ be a finite subgroup of the symmetry group of some nondegenerate quasihomogeneous polynomial
    in $\C[x_1,\dots,x_n]$.
    We define the \emph{age} of $g \in G$ as
    \begin{equation*}
      \age(g) = \frac{1}{2\pi\iu} \sum_{i=1}^{n} \log \lambda_i,
    \end{equation*}
    where $\lambda_1,\dots,\lambda_n$ are the eigenvalues of $g$
    and the branch of the logarithmic function is chosen
    to satisfy $0 \le \log{z} < 2\pi\iu$ for $z \in \Cx$ s.t. $|z| = 1$.
  \end{definition}

  For $g \in G$, we denote the eigenspace of $g$ with the eigenvalue $1$ by $V^{g}$, i.e. $V^{g} = \Ker(E_{n} - g)$.
  Here, $E_{n}$ means the identity matrix of rank $n$.
  Let $n_{g}$ be the dimension of $V^{g}$.

  \begin{lemma}
    \label{lem:age-ng-relation}
    Let $W(x_1, \dots, x_n)$ be a quasihomogeneous polynomial and $G$ be
    a finite subgroup of the symmetry group $G_{W}$.
    For every $g \in G$, we have
    \begin{equation*}
      \age(g) + \age(g^{-1}) = n - n_{g}.
    \end{equation*}
  \end{lemma}

  \begin{proof}
    Let $\lambda_1, \dots, \lambda_n$ be the eigenvalues of $g$.
    Then we have $\log\lambda_i^{-1} = 2\pi\iu - \log\lambda_i$ if $\lambda_i \neq 1$.
    Since the number of $i$'s with $\lambda_{i} \neq 1$ is $n - n_{g}$, we have $\age(g) + \age(g^{-1}) = n - n_{g}$.
  \end{proof}

  \begin{definition}
    Let $W(x_1, \dots, x_n)$ be a nondegenerate quasihomogeneous polynomial with charges $q_i$.
    We define the \emph{central charge} $\Hat{c}$ of $W$ by
    \begin{equation*}
      \Hat{c} = \sum_{i=1}^{n} (1-2q_i).
    \end{equation*}
  \end{definition}

  \begin{remark}
    This definition of the central charge $\Hat{c}$ is in the sense of $\mathcal{N} = 2$ supersymmetric conformal field theories.
    There is a relation with the ordinary central charge $c$ for conformal field theories:
    \begin{equation*}
      c = 3\Hat{c}.
    \end{equation*}
  \end{remark}

  \begin{definition}
    \label{def:calabi-yau-condition}
    The charges $q_{1}, \dots, q_{n}$ of a nondegenerate quasihomogeneous polynomial is said to satisfy
    the \emph{Calabi-Yau condition} if the equation
    \begin{equation*}
      q_{1} + \dots + q_{n} = 1
    \end{equation*}
    holds.
    Similarly, the condition
    \begin{equation*}
      q_{1} + \dots + q_{n} \in \Z_{\geq 0}
    \end{equation*}
    is called the \emph{generalized Calabi-Yau condition}.
  \end{definition}

  \begin{remark}
    The geometric interpretation of these conditions are as follows:
    If the Calabi-Yau condition is satisfied, $X_{W} = \set{W = 0}$ is a Calabi-Yau hypersurface
    in the weighted projective space $\P(w_{1}, \dots, w_{n})$~\cite{chiodo2011}.
    More generally, Borisov~\cite{borisov2013} proposed that the generalized Calabi-Yau condition
    is analogous to the Batyrev-Borisov construction for a Calabi-Yau complete intersection of $\sum_{i} q_{i}$ hypersurfaces.
  \end{remark}

  \begin{lemma}
    The charges $q_{1}, \dots, q_{n}$ satisfy the generalized Calabi-Yau condition if and only if
    the exponential grading operator $J$ is an element of $\SL_{n}(\C)$.
  \end{lemma}

  \begin{proof}
    Clear since the determinant of $J$ is $\exp \big( 2\pi\iu (q_{1} + \dots + q_{n}) \big)$.
  \end{proof}

\section{Landau-Ginzburg orbifold}
  \label{sec:landau-ginzburg-orbifold}
  In this section, we give the definition of the Landau-Ginzburg orbifold.
  Roughly speaking, a Landau-Ginzburg orbifold is a direct sum of Jacobi rings tensored with a certain one dimensional
  vector space.
  The construction of a Landau-Ginzburg orbifold as a bigraded vector space was given
  by Intriligator and Vafa~\cite{intriligator1990}.
  We define a Landau-Ginzburg orbifold following their construction and also check that this definition
  is well-defined, which was not discussed by them.

  Let $W(x_1, \dots, x_n)$ be a nondegenerate quasihomogeneous polynomial
  and $G$ be a finite subgroup of the symmetry group $G_{W}$.
  Let $y_1, \dots, y_n$ be eigenvectors of $g = (g_{i j})$ which correspond to the eigenvalues $\lambda_1, \dots, \lambda_n$.
  We can assume that $y_i$ has the same weight and charge with $x_i$ since $g_{i j} = 0$ if $q_{i} \neq q_{j}$.
  Let $W'$ be a polynomial such that $W'(y_1, \dots, y_n) = W(x_1, \dots, x_n)$.
  Denote by $I^{g} = \set{i_1, \dots, i_{n_{g}}}$ a set of indices of $y_{i}$ s.t. $y_{i_1}, \dots, y_{i_{n_{g}}}$ become a basis of $V^g = \Ker(E_{n} - g)$.

  If $n_{g}$ is greater than zero, $I^{g}$ is not the empty set.
  We define a polynomial $W^{g}$ in $R^{g} \defeq \C[y_{i_{1}}, \dots, y_{i_{n_{g}}}]$
  by $W^g(y_{i_1}, \dots, y_{i_{n_{g}}}) = W'(y_1, \dots, y_n)\vert_{R^g}$,
  that is, setting $y_i = 0$ if $\lambda_i \neq 1$.
  The restriction of $C_{G}(g)$ to $V^{g}$ makes $R^{g}$ a $C_{G}(g)$-module.
  If $n_{g}$ is zero, we regard $R^{g}$ as a trivial $C_{G}(g)$-module $\C$, $W^{g}$ as some constant number and $I_{\nabla W^{g}}$ as
  the zero ideal of $R^{g}$.

  \begin{lemma}
    \label{lem:sector-generator}
    If $n_{g}$ is greater than zero, $W^g$ is a nondegenerate quasihomogeneous polynomial and
    the centralizer $C_G(g)$ of $g$ is a finite subgroup of the symmetry group $G_{W^{g}}$.
  \end{lemma}

  \begin{proof}
    For nondegeneracy, it suffices to show that $R^{g}/I_{\nabla W^{g}}$ is finite dimensional.
    By the construction of $W^{g}$, we see that $W' - W^{g} \in (y_{i})_{i \notin I^{g}}$.
    In fact, we have $W' - W^{g} \subset (y_{i})_{i \notin I^{g}}^{2}$.
    We can write $W - W^{g}$ as
    \begin{equation*}
      W' - W^{g} = \sum_{i \notin I^{g}}y_{i}f_{i}(y_{1}, \dots, y_{n}).
    \end{equation*}
    Now, consider the action of $g$ on the above equation.
    The left hand side is invariant, and which shows that each $f_{i}$ must be an element of the ideal $(y_{i})_{i \notin I^{g}}$.
    Hence we have $W' - W^{g} \subset (y_{i})_{i \notin I^{g}}^{2}$.
    It follows that $(\partial W' / \partial y_{i})_{i \notin I^{g}} \subset (y_{i})_{i \notin I^{g}}$ and $(\partial (W'-W^{g}) / \partial y_{i})_{i \in I^{g}} \subset (y_i)^{2}_{i \notin I^{g}}$.
    Therefore we have
    \begin{equation*}
      (\partial W^{g} / \partial y_{i})_{i \in I^{g}} + (y_{i})_{i \notin I^{g}}
      = (\partial W' / \partial y_{i})_{i \in I^{g}} + (y_{i})_{i \notin I^{g}}
      \supset (\partial W' / \partial y_{i})_{i \in \set{1,\dots,n}}
    \end{equation*}
    which yields nondegeneracy as
    \begin{align*}
      \dim_{\C} R^{g}/I_{\nabla W^{g}} &= \dim_{\C} \C[y_{1}, \dots, y_{n}] / \big( (\partial W^{g} / \partial y_{i})_{i \in I^{g}} + (y_{i})_{i \notin I^{g}} \big) \\
      &\leq \dim_{\C} \C[y_{1}, \dots, y_{n}] / (\partial W' / \partial y_{i})_{i \in \set{1, \dots, n}} \\
      &= \dim_{\C} \C[y_{1}, \dots, y_{n}] / I_{\nabla W^{g}} < \infty.
    \end{align*}
    It is clear that $C_{G}(g)$ is finite group.
    Let $h$ be an element of $C_{G}(g) \subset G$.
    By definition of the symmetry group, we see $W'(h{\cdot}y_{1}, \dots, h{\cdot}y_{n}) = W'(y_{1}, \dots, y_{n})$.
    Restricting this equation to $y_{i} = 0$ for $i \notin I^{g}$, we have $W^{g}(h{\cdot}y_{i_{1}}, \dots, h{\cdot}y_{i_{n_{g}}}) = W^{g}(y_{i_{1}}, \dots, y_{i_{n_{g}}})$.
  \end{proof}

  Assume that $n_{g}$ is greater than zero.
  Let us consider the action of $h = (h_{i j}) \in C_{G}(g)$ to generators of $I_{\nabla W^{g}}$, namely $\partial W^{g} / \partial y_{i}$, $i \in I^{g}$.
  In general, if $f$ is a polynomial in $\C[x_{1}, \dots, x_{n}]$ and $g$ is an element of $\GL_{n}(\C)$,
  the action of $g$ and the derivative with respect to $x_{i}$ have a relation
  \begin{equation*}
    g {\cdot} \frac{\partial}{\partial x_{i}} f(x_{1}, \dots, x_{n}) = \sum_{j = 1}^{n} g_{j i}^{-1} \frac{\partial}{\partial x_{j}} (g {\cdot} f)(x_{1}, \dots, x_{n}).
  \end{equation*}
  Therefore we have
  \begin{equation*}
    h {\cdot} \frac{\partial W^{g}}{\partial y_{i}} = \sum_{j \in I^{g}} h_{j i}^{-1} \frac{\partial W^{g}}{\partial y_{j}}
  \end{equation*}
  by Lemma~\ref{lem:sector-generator}.
  This shows that $I_{\nabla W^{g}}$ is closed under the action of $C_{G}(g)$ and hence $R^{g}/I_{\nabla W^{g}}$ is
  also a $C_{G}(g)$-module.

  Let $\Omega_{g}$ be a one dimensional vector space generated by a symbol $d y_{i_{1}} \wedge \dots \wedge d y_{i_{n_{g}}}$.
  Define the action of $h \in C_{G}(g)$ on $\Omega_{g}$ by
  \begin{equation*}
    h {\cdot} (d y_{i_{1}} \wedge \dots \wedge d y_{i_{n_{g}}}) = \det(h \vert_{V^{g}}) d y_{i_{1}} \wedge \dots \wedge d y_{i_{n_{g}}}.
  \end{equation*}
  and extending linearly.
  $\Omega_{g}$ is equivalent to $\bigwedge\hspace{-3pt}^{n_{g}} V^{g}$ as a $G$-module.
  If $n^{g}$ is zero, we regard $\Omega_{g}$ as a one dimensional trivial representation of $C_{G}(g)$.

  The unprojected twisted sector of the Landau-Ginzburg orbifold for $W^{g}$ is the Jacobi ring of $W^{g}$ tensored with
  $\Omega_{g}$ and the (projected) twisted sector is its $C_{G}(g)$-invariant subspace.

  \begin{definition}
    \label{def:twisted-sector}
    For $g \in G$, the \emph{unprojected $g$ twisted sector} $\mathcal{Q}_{W, g}$ of a Landau-Ginzburg orbifold is defined to be the $\C$-vector space
    \begin{equation*}
      \mathcal{Q}_{W, g}
      = \mathcal{Q}_{W^{g}} \otimes_{\C} \Omega_{g}
      = R^{g}/I_{\nabla W^g} \otimes_{\C} \Omega_{g}.
    \end{equation*}
    The \emph{left charge} and \emph{right charge} of $f \otimes v \in \mathcal{Q}_{W, g}$ are defined by $\text{charge}(f) + \age(g) - \sum_{\lambda_{i} \neq 1}q_{i}$ and $\text{charge}(f) + \age(g^{-1}) - \sum_{\lambda_{i} \neq 1}q_{i}$ respectively.
    The left and right charge make $\mathcal{Q}_{W, g}$ a bigraded $\C$-vector space.
    In addition, we define the (projected) \emph{$g$ twisted sector} $\mathscr{H}_{W, G, g}$ as
    the invariant subspace of $\mathcal{Q}_{W, g}$ with respect to the action of the centralizer
    $C_G(g)$ of $g$, i.e.
    \begin{equation*}
      \mathscr{H}_{W, G, g}
      = \mathcal{Q}_{W, g}^{C_G(g)}
      = \big( R^{g}/I_{\nabla W^g} \otimes_{\C} \Omega_{g} \big)^{C_G(g)}.
    \end{equation*}
    Denote the homogeneous subspace with the left and right charge $(p, q)$ of $\mathscr{H}_{W, G, g}$ by $\mathscr{H}_{W, G, g}^{p, q}$.
    The \emph{Hodge number} $h_{g}^{p, q}(W, G)$ of the twisted sector $\mathscr{H}_{W, G, g}$ is defined by
    \begin{equation*}
      h_{g}^{p, q}(W, G) \defeq \dim_{\C} \mathscr{H}_{W, G, g}^{p, q}.
    \end{equation*}
  \end{definition}

  The left and right charge are well-defined.
  Indeed, the gradient ideal $I_{\nabla W^{g}}$ is a homogeneous ideal since each generator $\partial W^{g} / \partial y_{i}$ is
  a homogeneous polynomial with charge $1 - q_{i}$.

  \begin{example}
    The unprojected $E_{n}$-twisted sector is a tensor product of the Jacobi ring $\mathcal{Q}_{W}$ and $\C d x_{1} \wedge \dots \wedge d x_{n}$.
    If $G$ is a subgroup of $\SL_{n}(\C)$, $\Omega_{n}$ is a trivial representation of $C_{G}(E_{n}) = G$.
    In particular, $\mathscr{H}_{W, G, E_{n}}$ contains an element of the form $1 \otimes d x_{1} \wedge \dots \wedge d x_{n}$.
    The existence of this element seems to be important if we try to define a product structure on the Landau-Ginzburg orbifold
    since this is a candidate for the identity element~\cite{krawitz2010}.
    The left and right charges coincide and are specified by the charge of a polynomial
    since $\age(E_{n}) - \sum_{\lambda_{i} \neq 1} q_{i} = 0$.
  \end{example}

  \begin{example}
    Let $g$ be an element of $G$ such that $n_{g} = 0$.
    Then $R^{g}/I_{W^{g}}$ and $\Omega_{g}$ are both trivial $C_{G}(g)$-modules hence $\mathcal{Q}_{W, G, g}$ and $\mathscr{H}_{W, G, g}$ are also trivial.
    The left and right charge are $\age(g) - \sum_{i=1}^{n}q_{i}$ and $\age(g^{-1}) - \sum_{i=1}^{n} q_{i}$ respectively.
  \end{example}

  \begin{lemma}
    \label{lem:twisted_sector_conjugate_invariant}
    An (unprojected) twisted sector does not depend on a choice of a representative of a conjugacy class.
    More precisely, if $g$ and $g'$ are conjugate in $G$, then we have $\mathcal{Q}_{W, g} \simeq \mathcal{Q}_{W, g'}$
    and $\mathscr{H}_{W, G, g} \simeq \mathscr{H}_{W, G, g'}$ as bigraded vector spaces.
  \end{lemma}

  \begin{proof}
    From the assumption, there is an element $h$ in $G$ s.t. $h g h^{-1} = g'$.
    If $y_{1}, \dots, y_{n_{g}}$ is a basis of $V^{g}$, $y_{i}' = h {\cdot} y_{i}$, $1 \leq i \leq n_{g}$
    is a basis of $V^{g'}$.
    $h$ induces a bigrading preserving isomorphism $\mathcal{Q}_{W, g} \to \mathcal{Q}_{W, g'}$ by
    \begin{equation*}
      f \otimes d y_{1} \wedge \dots \wedge d y_{n_{g}} \mapsto \det(h^{-1} \vert_{V^{g}}) f' \otimes d y_{1}' \wedge \dots \wedge d y_{n_{g}}'
    \end{equation*}
    where $f'$ is a polynomial such that $f'(y_{1}', \dots, y_{n_{g}}') = f(y_{1}, \dots, y_{n_{g}})$.
    Furthermore, this isomorphism clearly commutes with an action of $C_{G}(g) \simeq C_{G}(g')$.
    Hence invariant subspaces are isomorphic, that is $\mathscr{H}_{W, G, g} \simeq \mathscr{H}_{W, G, g'}$.
  \end{proof}

  We have completed all preparations to define a Landau-Ginzburg orbifold.
  The Landau-Ginburg orbifold of the pair $(W, G)$ is a direct sum over conjugacy classes of $G$ of the projected
  twisted sectors.
  By Lemma~\ref{lem:twisted_sector_conjugate_invariant}, this does not depend on a choice of representatives of conjugacy
  classes.

  \begin{definition}
    Let $W(x_1, \dots, x_n)$ be a nondegenerate quasihomogeneous polynomial with unique weights,
    $G$ be a finite subgroup of the symmetry group $G_{W}$.
    and $S \subset G$ be a set of representatives of the conjugacy classes of $G$.
    The \emph{Landau-Ginzburg orbifold} $\mathscr{H}_{W, G}$ for the pair $(W, G)$ is defined by
    \begin{equation*}
      \mathscr{H}_{W, G}
        = \bigoplus_{g \in S} \mathscr{H}_{W, G, g}
        = \bigoplus_{g \in S} \big(R^{g} / I_{\nabla W^{g}} \otimes_{\C} \Omega_{g} \big)^{C_G(g)}.
    \end{equation*}
    The bigrading of $\mathscr{H}_{W, G}$ is induced from that of each summand.
    We denote the homogeneous subspace with the left and right charge $(p, q)$ by $\mathscr{H}_{W, G}^{p, q}$.
    We have a decomposition of the $\C$-vector space
    \begin{equation*}
      \mathscr{H}_{W, G} = \bigoplus_{p, q \in \Q} \mathscr{H}_{W, G}^{p, q}.
    \end{equation*}
    The \emph{Hodge number} $h^{p, q}(W, G)$ of the pair $(W, G)$ is defined to be
    \begin{equation*}
      h^{p, q}(W, G) \defeq \dim_{\C} \mathscr{H}_{W, G}^{p, q}
    \end{equation*}
    where $p$ and $q$ are rational numbers.
  \end{definition}

  \begin{lemma}
    \label{lem:twisted_sector_similar_invariant}
    Let $x_{1}, \dots, x_{n}$ and $x_{1}', \dots, x_{n}'$ are different coordinate systems, namely
    $x_{i}' = \phi {\cdot} x_{i}$ where $\phi$ is an element of $\GL_{n}(\C)$.
    Put $W'(x'_{1}, \dots, x'_{n}) = W(x_{1}, \dots, x_{n})$ and $G' = \phi G \phi^{-1}$.
    Then we have $\mathcal{Q}_{W, g} \simeq \mathcal{Q}_{W', \phi g \phi^{-1}}$ and
    $\mathscr{H}_{W, G, g} \simeq \mathscr{H}_{W', G', \phi g \phi^{-1}}$ as bigraded vector spaces for any $g \in G$.
    In other words, an (unprojected) twisted sector does not depend on a choice of the coordinates.
  \end{lemma}

  \begin{proof}
    The proof is almost the same with Lemma~\ref{lem:twisted_sector_conjugate_invariant}.
    If $y_{1}, \dots, y_{n_{g}}$ is a basis of $V^{g}$,
    An isomorphism $\mathcal{Q}_{W, g} \to \mathcal{Q}_{W', \phi g \phi}$ is given by
    \begin{equation*}
      f \otimes d y_{1} \wedge \dots \wedge d y_{n_{g}} \mapsto \det(\phi \vert_{V^{g}}) f' \otimes d y_{1}' \wedge \dots \wedge d y_{n_{g}}'
    \end{equation*}
    where $y_{i}' = \phi {\cdot} y_{i}$ and $f'$ is a polynomial s.t. $f'(y_{1}', \dots, y_{n_{g}}') = f(y_{1}, \dots, y_{n_{g}})$.
  \end{proof}

  The following proposition is generalization of the property of a Jacobi ring stated in Remark~\ref{rem:jacobi-ring-invariant-singularity}

  \begin{proposition}
    A Landau-Ginzburg orbifold does not depend on a choice of the coordinates.
  \end{proposition}

  \begin{proof}
    Each projected twisted sector is invariant under a linear change of the coordinates by Lemma~\ref{lem:twisted_sector_similar_invariant},
    which follows the result.
  \end{proof}

  \begin{example}
    \label{ex:type_An_orbifold}
    As an example, consider the Landau-Ginzburg orbifold of type $A_{n}$ singularities (recall Example~\ref{ex:type-An}).
    Let $W = x^{n}$ for $n \geq 2$.
    An element of $\Cx$ acts on $x$ by scalar multiplication, and we see that the maximal symmetry group $G_{W}$ is
    \begin{equation*}
      G_{W} = \set{ z \in \Cx | z^{n} = 1} = \set{\zeta_{n}^{i} | 0 \leq i \leq n - 1}
    \end{equation*}
    where $\zeta_{n}$ means the primitive $n$-th root of unity.
    Let us construct the Landau-Ginzburg orbifold of the pair $(W, G_{W})$.
    Since $G_{W}$ is an abelian group, any conjugacy class consists of one element and a centralizer is whole $G_{W}$
    for an arbitrary element.
    $n_{g}$ is $1$ if $g = 1$ and $0$ otherwise, therefore we have the unprojected twisted sectors
    \begin{equation*}
      \mathcal{Q}_{W, g} =
      \begin{cases}
        \C[x]/(x^{n-1}) \otimes \C d x& \text{if $g = 1$,} \\
        \Omega_{g} \simeq \C & \text{otherwise.} \\
      \end{cases}
    \end{equation*}
    The $G_{W}$-invariant subspace of these vector spaces are the projected twisted sectors $\mathscr{H}_{W, G_{W}, g}$.
    $\mathcal{Q}_{W, 1}$ has a basis of the form $x^{i} \otimes d x$, $0 \leq i \leq n-2$ which diagonalize
    the action of $G_{W}$.
    Observing
    \begin{equation*}
      \zeta_{n} {\cdot} (x^{i} \otimes d x) = \zeta_{n}^{i+1} x^{i} \otimes d x \neq x^{i} \otimes d x
    \end{equation*}
    for all $0 \leq i \leq n - 2$, we obtain $\mathscr{H}_{W, G_{W}, 1} = 0$.
    For $g \neq 1$, the projected twisted sectors are the same as the unprojected ones.
    From these result, it follows that the Landau-Ginzburg orbifold $\mathscr{H}_{W, G_{W}}$ for the pair $(W, G_{W})$ is
    \begin{equation*}
      \mathscr{H}_{W, G_{W}} = \bigoplus_{i=1}^{n-1} \Omega_{\zeta_{n}^{i}} \simeq \bigoplus_{i=1}^{n-1} \C.
    \end{equation*}
    The Hodge numbers $h^{p, q}(W, G_{W})$ are
    \begin{equation*}
      h^{p, q}(W, G_{W}) =
      \begin{cases}
        1 & \text{if $(p, q) = \big(\frac{i-1}{n}, \frac{n-i-1}{n} \big)$ for $1 \leq i \leq n-1$,} \\
        0 & \text{otherwise.} \\
      \end{cases}
    \end{equation*}

    If $n$ is not a prime, $G_{W}$ has a nontrivial subgroup.
    We will calculate the Hodge numbers of the Landau-Ginzburg orbifolds of such subgroups in Example~\ref{ex:type_An_Hodge_number}
    using the formula of the generating function of the Hodge numbers proved in the next section.
  \end{example}

\section{Properties of the Hodge numbers}
  \label{sec:generating-function}
  In the previous section, we have defined the Landau-Ginzburg orbifold and its Hodge numbers.
  In this section, we prove a formula for the generating function of these Hodge numbers.
  Using this formula, we see some properties of Hodge numbers.

  \begin{definition}
    Let $W$ be a nondegenerate quasihomogeneous polynomial with unique weights and $G$ be a finite subgroup of
    the symmetry group $G_{W}$.
    We define the \emph{Poincaré polynomial} of $\mathscr{H}_{W, G, g}$ for $g \in G$ by
    \begin{equation*}
      P_{g}(W, G; u, v) = \sum_{p,q \in \Q} h_{g}^{p,q}(W, G) u^{p} v^{q}
    \end{equation*}
    and similarly the \emph{Poincaré polynomial} of the pair $(W, G)$ by
    \begin{equation*}
      P(W, G; u, v) = \sum_{p,q \in \Q} h^{p,q}(W, G) u^p v^q.
    \end{equation*}
  \end{definition}

  \begin{proposition}
    \label{prop:proposition-for-generating-function}
    The Poincaré polynomial $P_{E_{n}}(W, G; u, v)$ of $\mathscr{H}_{W, G, E_{n}}$ is given by
    \begin{equation*}
      P_{E_{n}}(W, G; u, v) = \frac{1}{\#G} \sum_{g \in G} \prod_{i=1}^{n}
        \frac{\lambda_i - (u v)^{1-q_{i}}}{1 - \lambda_i(u v)^{q_i}}
    \end{equation*}
    where $\lambda_{i}$ is an eigenvalue of $g$ belonging to an eigenvector with the charge $q_{i}$.
  \end{proposition}

  \begin{proof}
    Let $R = \C[x_{1}, \dots, x_{n}]$ and $V = \C x_{1} \oplus \dots \oplus \C x_{n}$.
    By definition, $\mathscr{H}_{W, G, E_{n}}$ is a fixed subspace of the tensor product of
    Jacobi ring $R/I_{\nabla W}$ of $W$ and $\Omega_{E_{n}}$.
    Since the generators $(\partial W / \partial x_{1}, \dots, \partial W / \partial x_{n})$
    of $I_{\nabla W}$ form a regular sequence, we have the exact sequence of $R$-modules
    \begin{equation*}
      0 \to \ExteriorProduct{n}{R^{n}} \xrightarrow{d_{n}} \dots \xrightarrow{d_{2}} R^{n}
      \xrightarrow{d_{1}} R \xrightarrow{\pi} R/I_{\nabla W} \to 0.
    \end{equation*}
    from the Koszul complex.
    Here $R^{n}$ is a free $R$-module generated by the symbols $e_{1}, \dots, e_{n}$,
    $\pi$ is a canonical projection and $d_{p}$ sends $e_{i_{1}} \wedge \dots \wedge e_{i_{p}}$ to
    \begin{equation*}
      \sum_{k=1}^{p}(-1)^{k+1}\frac{\partial W}{\partial x_{i_{p}}} e_{i_{1}} \wedge \dots \wedge
      \widehat{e_{i_{p}}} \wedge \dots \wedge e_{i_{n}}
    \end{equation*}
    where the hat means the omission of the symbol.
    Now we would like to regard this as the exact sequence of $G$-modules.
    Define the action of $g \in G$ to $e_{i}$ by $g {\cdot} e_{i} = \sum_{j = 1}^{n} (g^{-1})_{j i} e_{j}$.
    Note that this action is the same with the one to $\partial W / \partial x_{i}$.
    Therefore we have the exact sequence of $G$-modules.
    Tensoring $\Omega_{E_{n}}$ to all objects, we have the exact sequence of $G$-modules
    \begin{equation*}
      0 \to \ExteriorProduct{n}{R^{n}} \otimes \Omega_{E_{n}} \to \dots \to
      R^{n} \otimes \Omega_{E_{n}} \to R \otimes \Omega_{E_{n}} \to
      \mathcal{Q}_{W, E_{n}} \to 0.
    \end{equation*}
    Next, we define the degree of each elements by $\degree(x_{i}) = w_{i}$ and $\degree(e_{i}) = d - w_{i}$.
    Then each morphism is a degree preserving map and each object is a $\Z_{\geq 0}$-graded $\C$-vector space.
    Note that this degree for $\mathscr{H}_{W, G, E_{n}}$ is the one of
    Definition~\ref{def:twisted-sector} times $d$.
    Restricting to $G$-invariant subspaces, we have an exact sequence of $\Z_{\geq 0}$-graded $\C$-vector spaces
    \begin{equation*}
      0 \to (\ExteriorProduct{n}{R^{n}} \otimes \Omega_{E_{n}})^{G} \to \dots \to
      (R^{n} \otimes \Omega_{E_{n}})^{G} \to (R \otimes \Omega_{E_{n}})^{G} \to
      \mathscr{H}_{W, G, E_{n}} \to 0.
    \end{equation*}
    Denote the Hilbert function by $H(X, \cdot) \colon \Z_{\geq 0} \to \Z_{\geq 0}$ and the Hilbert series
    by $F(X, t) = \sum_{k = 0}^{\infty} H(X, k) t^{k}$ for a $\Z_{\geq 0}$-graded vector space $X$.
    If $X$ is a $G$-module, we see from the rudiments of the representation theory of finite groups
    \begin{equation*}
      H(X, k) = \frac{1}{\# G} \sum_{g \in G} \tr_{X_{k}}(g)
    \end{equation*}
    where $X_{k}$ means the degree $k$ component of $X$.
    By definition of the degree, we have
    \begin{align*}
      & F((\ExteriorProduct{p}{R^{n}} \otimes \Omega_{E_{n}})^{G}, t) \\
      & = \frac{1}{\# G} \sum_{g \in G} \frac{\lambda_{1} \dots \lambda_{n}}{(1-\lambda_{1} t^{w_{1}})
        \dots (1 - \lambda_{n} t^{w_{n}})} \sum_{i_{1} < \dots < i_{p}} t^{\sum_{j = 1}^{p}(d - w_{i_{j}})}
        \lambda_{i_1}^{-1} \dots \lambda_{i_{p}}^{-1}
    \end{align*}
    for each $p$ where $\lambda_{i}$ is an eigenvalue of $g$ belonging an eigenvector with charge $q_{i}$
    (this equation is also valid for $p = 0$ regarding the last sum as $1$).
    Therefore we have
    \begin{align*}
      & F(\mathscr{H}_{W, G, E_{n}}, t) \\
      & = \sum_{p = 0}^{n} (-1)^{p} F((\ExteriorProduct{p}{R^{n}} \otimes \Omega_{E_{n}})^{G}, t) \\
      & = \frac{1}{\# G} \sum_{g \in G} \frac{\lambda_{1} \dots \lambda_{n}}{(1-\lambda_{1} t^{w_{1}})
        \dots (1 - \lambda_{n} t^{w_{n}})} \sum_{p = 0}^{n}
        \sum_{i_{1} < \dots < i_{p}} (-1)^{p} t^{\sum_{j = 0}^{p}(d - w_{i_{j}})}
        \lambda_{i_1}^{-1} \dots \lambda_{i_{p}}^{-1} \\
      & = \frac{1}{\# G} \sum_{g \in G} \frac{\lambda_{1} \dots \lambda_{n}}{(1-\lambda_{1} t^{w_{1}})
        \dots (1 - \lambda_{n} t^{w_{n}})} (1-\lambda_{1}^{-1} t^{d - w_{1}}) \dots (1 - \lambda_{n}^{-1} t^{d-w_{n}}) \\
      & = \frac{1}{\# G} \sum_{g \in G} \frac{(\lambda_{1} - t^{d - w_{1}}) \dots (\lambda_{n} - t^{d-w_{n}})}
        {(1-\lambda_{1} t^{w_{1}}) \dots (1 - \lambda_{n} t^{w_{n}})}
        = \frac{1}{\# G} \sum_{g \in G} \prod_{i=1}^{n} \frac{\lambda_{i} - t^{d - w_{i}}}{1-\lambda_{i} t^{w_{i}}}.
    \end{align*}
    By definition of the charge of the twisted sectors, it follows from the above equation that
    \begin{equation*}
      P_{E_{n}}(W, G; u, v) = \frac{1}{\# G} \sum_{g \in G} \prod_{i=1}^{n} \frac{\lambda_{i} - (u v)^{1 - q_{i}}}{1-\lambda_{i} (u v)^{q_{i}}}. \qedhere
    \end{equation*}
  \end{proof}

  Let $g \in G$ and $h \in C_{G}(g)$.
  Since $g$ and $h$ commute, these can be simultaneously diagonalizable.
  Denote the eigenvalues of $h$ by $\lambda_{1}^{h, g}, \dots, \lambda_{n}^{h, g}$.
  The indices of these eigenvalues can be chosen so that the eigenvector corresponding to $\lambda_{i}^{h, g}$ has
  charge $q_{i}$.
  Note that we can assume that $\lambda_{i}^{h, g}$ satisfies $(\lambda_{i}^{h, g})^{-1} = \lambda_{i}^{h^{-1}, g}$.

  \begin{theorem}
    \label{thm:generalized-vafas-formula}
    Let $W \in \C[x_1, \dots, x_n]$ be a nondegenerate quasihomogeneous polynomial with unique weights,
    $G$ be a finite subgroup of the symmetry group $G_{W}$ and $S \subset G$ be any set of representatives
    of the conjugacy classes of $G$.
    The Poincaré polynomial $P(W, G; u, v)$ of the pair $(G, W)$ is given by
    \begin{equation*}
      P(W, G; u, v) = \sum_{g \in S} P_{g}(W, G; u, v)
    \end{equation*}
    where
    \begin{equation*}
      P_{g}(W, G; u, v) = \frac{1}{\#C_G(g)}
        u^{\age(g)}v^{\age(g^{-1})}(u v)^{-\sum_{\lambda_i^g \neq 1}q_i}
        \sum_{h \in C_G(g)} \prod_{\lambda_{i}^{g} = 1}
        \frac{\lambda_i^{h, g} - (u v)^{1-q_i}}{1 - \lambda_i^{h, g}(u v)^{q_i}}.
    \end{equation*}
  \end{theorem}

  \begin{proof}
    By definition of the left and right charge of Landau-Ginzburg orbifolds, we have
    \begin{equation*}
      P(W, G; u, v) = \sum_{g \in S} u^{\age(g)}v^{\age(g^{-1})}(u v)^{-\sum_{\lambda_i^g \neq 1}q_i} P_{E_{n_{g}}}(W^{g}, C_{G}(g); u, v)
    \end{equation*}
    which proves the theorem.
  \end{proof}

  The formula of the Poincaré series in Theorem~\ref{thm:generalized-vafas-formula} indicates that
  the Hodge numbers $h^{p, q}(W, G)$ do not depend on a precise form of $W$.
  Only the existence of $W$ is required to apply Theorem~\ref{thm:generalized-vafas-formula} to $G$.

  \begin{example}
    \label{ex:type_An_Hodge_number}
    Let us calculate all possible Hodge numbers of the orbifolds of type $A_{n}$ using the formula above.
    Recall the notation in Example~\ref{ex:type_An_orbifold}.
    Notice that $G_{W}$ has a nontrivial subgroup if $n$ is not a prime number.
    Let $G$ be a subgroup of $G_{W}$ generated by $\zeta_{n}^{l}$ where $l \mid n$.
    The order of $G$ is $n/l$ and an element of $G$ can be written as $\zeta_{n}^{l i}$ for $0 \leq i \leq n - 1$.
    For the twisted sector of $1$, we see
    \begin{equation*}
      P_{1}(W, G; u, v) = \frac{1}{\# G} \sum_{i = 0}^{n/l - 1} \frac{\zeta_{n}^{l i} - (u v)^{1-\frac{1}{n}}}{1 - \zeta_{n}^{l i} (u v)^{\frac{1}{n}}}.
    \end{equation*}
    Notice that $(\zeta_{n}^{l i} - t^{1-\frac{1}{n}})/(1 - \zeta_{n}^{l i}t^{\frac{1}{n}})$, $0 \leq i \leq n/l - 1$
    are the roots of the equation
    \begin{equation*}
      (X + t^{1-\frac{1}{n}})^{n/l} - (t^{\frac{1}{n}}X+1)^{n/l} = 0
    \end{equation*}
    and hence the sum of the roots are
    \begin{equation*}
      \sum_{i = 0}^{n/l - 1} \frac{\zeta_{n}^{l i} - t^{1-\frac{1}{n}}}{1 - \zeta_{n}^{l i}t^{\frac{1}{n}}}
      = -\frac{n}{l} \frac{t^{1-\frac{1}{n}} - t^{\frac{1}{l} - \frac{1}{n}}}{1 - t^{\frac{1}{l}}} =
      \begin{dcases*}
        0 & if $l = 1$, \\
        \frac{n}{l} t^{1 - \frac{1}{n}} \sum_{i = 1}^{l-1} t^{-\frac{i}{l}} & if $l \geq 2$. \\
      \end{dcases*}
    \end{equation*}
    Therefore we have
    \begin{equation*}
      P_{1}(W, G; u, v) = (u v)^{1 - \frac{1}{n}} \sum_{i = 1}^{l} (u v)^{-\frac{i}{l}} - (u v)^{-\frac{1}{n}}.
    \end{equation*}
    For $\zeta_{n}^{l i} \in G$, $1 \leq i \leq n/l - 1$, the twisted sector is one dimensional and we see
    \begin{equation*}
      P_{\zeta_{n}^{l i}}(W, G; u, v) = u^{\frac{l i}{n}} v^{1 - \frac{l i}{n}} (u v)^{-\frac{1}{n}}
      = u^{\frac{l i - 1}{n}} v^{\frac{n - l i - 1}{n}}.
    \end{equation*}
    From the above calculations, we see that the Poincaré polynomial of the pair $(W, G)$ is
    \begin{equation*}
      P(W, G; u, v) = (u v)^{1 - \frac{1}{n}} \sum_{i = 1}^{l} (u v)^{-\frac{i}{l}} +
      \sum_{i = 1}^{n/l} u^{\frac{l i - 1}{n}} v^{\frac{n - l i - 1}{n}}
      - (u v)^{-\frac{1}{n}} - u^{\frac{n-1}{n}} v^{-\frac{1}{n}}.
    \end{equation*}
    This equation is also valid for $l = n$ which means that $G$ is a trivial group $\set{1}$.
    For example, consider the case $l = 1$, i.e. $G = G_{W}$.
    Then we have
    \begin{equation*}
      P(W, G_{W}; u, v) = \sum_{i = 1}^{n - 1} u^{\frac{i - 1}{n}} v^{\frac{n - i - 1}{n}}
    \end{equation*}
    which indeed reproduces the result of Example~\ref{ex:type_An_orbifold}.

    We can observe an interesting duality from this Poincaré polynomial.
    Let $\check{G}$ be a subgroup of $G_{W}$ generated by $\zeta_{n}^{n/l}$.
    An easy calculation yields the relation
    \begin{equation*}
      P(W, G; u, v) = u^{\Hat{c}} P(W, \check{G}; u^{-1}, v).
    \end{equation*}
    In terms of the Hodge numbers this means
    \begin{equation*}
      h^{p, q}(W, G) = h^{\Hat{c} - p, q}(W, \check{G}) \quad \text{for $p, q \in \Q$,}
    \end{equation*}
    hence the Landau-Ginzburg orbifold of $(W, \check{G})$ is a mirror model of the orbifold of $(W, G)$ and vice versa.
    The type $A_{n}$ singularity is said to be ``self mirror'' from this property.
   \end{example}

  The following two lemmas prove that the Hodge numbers of a Landau-Ginzburg orbifold have the relations
  which are analogous to Hodge symmetry and Serre duality of the Hodge numbers of a $\Hat{c}$-dimensional compact Kähler manifold.
  These explain the reason why we call $h^{p, q}(W, G)$ the Hodge number.

  \begin{corollary}
    \label{cor:charge-conjugate}
    With the same notation with Theorem~\ref{thm:generalized-vafas-formula}, $P(W, G; u, v)$
    satisfies
    \begin{equation*}
      P(W, G; u, v) = P(W, G; v, u).
    \end{equation*}
    Equivalently, the Hodge numbers of $(W, G)$ satisfies
    \begin{equation*}
      h^{p,q}(W, G) = h^{q,p}(W, G) \quad \text{for} \quad p,q \in \Q.
    \end{equation*}
  \end{corollary}

  \begin{proof}
    If $S \subset G$ is a set of representatives of conjugacy classes of $G$,
    then $\Tilde{S} := \set{g^{-1} | g \in S}$ is also a set of representatives.
    By easy calculation, we have
    \begin{equation*}
      P_{g}(W, G; v, u) = P_{g^{-1}}(W, G, u, v)
    \end{equation*}
    and hence
    \begin{equation*}
      P(W, G; v, u) = \sum_{g \in S} P_{g}(W, G; v, u) = \sum_{g \in \Tilde{S}} P_{g}(W, G; u, v) = P(W, G; u, v). \qedhere
    \end{equation*}
  \end{proof}

  \begin{corollary}
    \label{cor:serre-duality}
    With the same notation with Theorem~\ref{thm:generalized-vafas-formula}, $P(u, v)$
    satisfies
    \begin{equation*}
      P(W, G; u, v) = (u v)^{\Hat{c}}P(W, G; u^{-1}, v^{-1}).
    \end{equation*}
    Equivalently, the Hodge numbers of $(W, G)$ satisfies
    \begin{equation*}
      h^{p,q}(W, G) = h^{\Hat{c} - p, \Hat{c} - q}(W, G) \quad \text{for} \quad p,q \in \Q.
    \end{equation*}
  \end{corollary}

  \begin{proof}
    First, consider the factor $(u v)^{\Hat{c}}u^{-\age(g)}v^{-\age(g^{-1})}(u v)^{\sum_{\lambda_{i}^{g} \neq 1}q_{i}}$.
    For the exponent of $u$, we have
    \begin{align*}
      -\age(g) + \sum_{\lambda_{i}^{g} \neq 1}q_{i} + \sum_{i}(1-2q_{i})
      &= n - \age(g) + \sum_{\lambda_{i}^{g} \neq 1}q_{i} - 2\sum_{i}q_{i} \\
      &= \age(g^{-1}) + n_{g} - \sum_{\lambda_{i}^{g} \neq 1}q_{i} - 2\sum_{\lambda_{i}^{g} = 1}q_{i} \\
      &= \age(g^{-1}) - \sum_{\lambda_{i}^{g} \neq 1}q_{i} + \sum_{\lambda_{i}^{g} = 1}(1-2q_{i})
    \end{align*}
    Similarly, we have $\age(g) - \sum_{\lambda_{i}^{g} \neq 1}q_{i} + \sum_{\lambda_{i}^{g} = 1}(1-2q_{i})$ for the exponent of $v$.
    Hence multiplying a part of the above factor to each summand, we have
    \begin{align*}
      (u v)^{\sum_{\lambda_{i}^{g} = 1}(1-2q_{i})} \prod_{\lambda_i^{g} = 1} \frac{\lambda_{i}^{h,g}-(u v)^{-1+q_{i}}}{1-\lambda_{i}^{h,g}(u v)^{-q_i}}
      &= \prod_{\lambda_i^{g}=1} (u v)^{1-2q_{i}} \frac{1-(\lambda_{i}^{h,g})^{-1}(u v)^{-1+q_{i}}}{(\lambda_{i}^{h,g})^{-1}-(u v)^{-q_i}} \\
      &= \prod_{\lambda_i^{g}=1} \frac{(u v)^{1-q_{i}}-(\lambda_{i}^{h,g})^{-1}}{(\lambda_{i}^{h,g})^{-1}(u v)^{q_i}-1} \\
      &= \prod_{\lambda_i^{g}=1} \frac{\lambda_{i}^{h^{-1},g}-(u v)^{1-q_{i}}}{1-\lambda_{i}^{h^{-1},g}(u v)^{q_i}}.
    \end{align*}
    Therefore we have
    \begin{align*}
      & (u v)^{\Hat{c}}P_{g}(W, G; u^{-1}, v^{-1}) \\
      &= \frac{1}{\#C_{G}(g)}(u v)^{\Hat{c}}u^{-\age(g)}v^{-\age(g^{-1})} (u v)^{\sum_{\lambda_{i}^{g} \neq 1}q_{i}} \sum_{h \in C_{G}(g)}
           \prod_{\lambda_i^{g} = 1} \frac{\lambda_{i}^{h,g}-(u v)^{-1+q_{i}}}{1-\lambda_{i}^{h,g}(u v)^{-q_i}} \\
      &= \frac{1}{\#C_{G}(g)}u^{\age(g^{-1})}v^{\age(g)} (u v)^{-\sum_{\lambda_{i} \neq 1} q_{i}}\sum_{h \in C_{G}(g)}
           \prod_{\lambda_i^{g}=1} \frac{\lambda_{i}^{h^{-1},g}-(u v)^{1-q_{i}}}{1-\lambda_{i}^{h^{-1},g}(u v)^{q_i}} \\
      &= P_{g^{-1}}(W, G; u, v).
    \end{align*}
    The sum of this relation with respect to $g$ over $S$ gives the desired result.
  \end{proof}

  If the condition $h^{p, q}(W, G) \neq 0$ implies $p - q \in \Z$, we can define the \emph{Witten index}
  $\tr(-1)^{F} = P(W, G; -1, -1)$.
  Physically, the Witten index is the difference between the number of bosonic and fermionic zero energy states of the corresponding
  conformal field theory.
  The Witten index can be considered as a Landau-Ginzburg orbifold counterpart for the Euler number.

\section{Orbifolds of nonsingular quintic threefolds}
  \label{sec:orbifold-quintic-threefold}
  In this section, we review the result of Yu~\cite{Yu_2016} which calculated the orbifold Hodge numbers of
  pairs of a nonsingular quintic threefold and a subgroup of the automorphism group which fixes a nowhere
  vanishing holomorphic $3$-form.

  First of all, we recall the definition of the orbifold Hodge numbers and orbifold Euler number,
  mainly following Batyrev-Dais~\cite{batyrev1996} and Yu~\cite{Yu_2016}.
  Let $X$ be a nonsingular compact Kähler manifold of dimension $n$ over $\C$
  and $\Tilde{G}$ be a finite subgroup of the automorphism group $\Aut(X)$ of $X$ which fixes a nowhere vanishing holomorphic $n$-form on $X$.
  For any $g \in \Tilde{G}$, we set $X^{g} = \set{p \in X | g{\cdot}p = p}$ and $X^{g} = X_{1}(g) \sqcup \dots \sqcup X_{r_{g}}(g)$
  its decomposition to the nonsingular connected components.
  The action of the centralizer $C_{G}(g)$ of $g$ can be restricted to $X^{g}$.
  For any point $p \in X^{g}$, the eigenvalues of $g$ in the holomorphic tangent space $T_{p}X$ are roots of unity:
  \begin{equation*}
    e^{2\pi\iu \alpha_{1}}, \dots, e^{2\pi\iu \alpha_{n}}.
  \end{equation*}
  where $\alpha_{i} \in [0, 1) \cap \Q$ are locally constant function on $X^{g}$.
  Define the locally constant function $\age(g) \colon X^{g} \to \Q_{\geq 0}$ on $X^{g}$ by
  \begin{equation*}
    \age(g)(p) = \sum_{i = 1}^{n} \alpha_{i} = \frac{1}{2\pi\iu} \tr_{T_{p}X}\log(g).
  \end{equation*}
  Notice that the age is in general not a constant unlike the Landau-Ginzburg case.
  We denote by $h_{g}^{p, q}(X, \Tilde{G})$ the sum of the dimensions of the $C_{\Tilde{G}}(g)$-invariant subspaces of $H^{p - \age(g), q - \age(g)}(X_{i}(g))$, i.e.
  \begin{equation*}
    h_{g}^{p, q}(X, \Tilde{G}) = \sum_{i = 1}^{r_{g}} \dim_{\C} \big( H^{p - \age(g), q - \age(g)}(X_{i}(g)) \big)^{C_{\Tilde{G}}(g)}.
  \end{equation*}
  Here, $\age(g)$ is a constant evaluated at the corresponding connected component $X_{i}(g)$.
  The \emph{orbifold Hodge numbers} $h^{p, q}(X, \Tilde{G})$ of $X/\Tilde{G}$ are defined by
  \begin{equation*}
    h^{p, q}(X, \Tilde{G}) = \sum_{g \in S} h_{g}^{p, q}(X, \Tilde{G})
  \end{equation*}
  where $S$ is a set of representatives of the conjugacy classes of $\Tilde{G}$.
  The \emph{orbifold Euler number} $e(X, \Tilde{G})$ of $X/\Tilde{G}$ is defined by
  \begin{equation*}
    e(X, \Tilde{G}) = \frac{1}{\# \Tilde{G}} \sum_{g h = h g}e(X^{g} \cap X^{h})
  \end{equation*}
  where $e$ in the right hand side means the Euler number of $X^{g} \cap X^{h}$.

  Yu~\cite{Yu_2016} calculated the possible orbifold Hodge numbers for nonsingular quintic threefolds.

  \begin{theorem}[Yu~\cite{Yu_2016}]
    \label{thm:nonsingular-quintic-threefold-hodge-number}
    Let $X$ be a nonsingular quintic threefold and $\Tilde{G}$ be a subgroup of the automorphism group $\Aut(X)$ which fixes
    a nowhere vanishing holomorphic $3$-form on $X$.
    Then $\Tilde{G}$ is one of the groups listed in Table~\ref{table:quintic-threefold-orbifold} of Appendix~\ref{sec:table}
    up to a linear change of the coordinates.
    The orbifold Hodge numbers $h^{p, q}(X, \Tilde{G})$ satisfy
    \begin{align*}
      h^{0, 0}(X, \Tilde{G}) = h^{3, 0}(X, \Tilde{G}) &= h^{0, 3}(X, \Tilde{G}) = h^{3, 3}(X, \Tilde{G}) = 1, \\
      h^{1, 1}(X, \Tilde{G}) = h^{2, 2}(X, \Tilde{G}) &, \quad h^{2, 1}(X, \Tilde{G}) = h^{1, 2}(X, \Tilde{G}),
    \end{align*}
    and $0$ otherwise.
    $h^{1, 1}(X, \Tilde{G})$, $h^{2, 1}(X, \Tilde{G})$ and the orbifold Euler number $e(X, \Tilde{G})$ are values specified in the same table.
  \end{theorem}

  From the above theorem, we can observe that the orbifold Hodge numbers $h^{p, q}(X, \Tilde{G})$ and the orbifold Euler number $e(X, \Tilde{G})$
  have the same type with Calabi-Yau threefolds.
  This is not an accidental phenomenon.
  Indeed, the quotient space $X/\Tilde{G}$ has a crepant resolution $Z$ which is again a Calabi-Yau threefold~\cite{bridgeland2001}.
  The orbifold Hodge numbers and Euler number of $X/\Tilde{G}$ coincide with the Hodge numbers $h^{p, q}(Z)$
  and the Euler number $e(Z)$ of $Z$ due to the McKay correspondence~\cite{batyrev1996}, i.e.
  \begin{equation*}
    h^{p, q}(X, \Tilde{G}) = h^{p, q}(Z), \quad e(X, \Tilde{G}) = e(Z).
  \end{equation*}
  The existence of $Z$ also reduces difficulty to determine the orbifold Hodge numbers.
  For example, this existence gives us the relation
  \begin{equation*}
    e(X, \Tilde{G}) = 2(h^{1, 1}(X, \Tilde{G}) - h^{2, 1}(X, \Tilde{G}))
  \end{equation*}
  from which we can determine $h^{2, 1}(X, \Tilde{G})$ without knowing the representation theory of $G$ of the cohomology of $X$
  if we have $h^{1, 1}(X, \Tilde{G})$ and $e(X, \Tilde{G})$ which are often easier to determine.
  Using such relations, Yu completed the computations of $h^{1, 1}(X, \Tilde{G})$, $h^{2, 1}(X, \Tilde{G})$ and $e(X, \Tilde{G})$.

\section{Hodge numbers of Landau-Ginzburg B-models}
  \label{sec:quintic-landau-ginzburg-orbifold}
  In this section, we calculate the Hodge numbers of Landau-Ginzburg orbifolds which satisfy the (generalized) Calabi-Yau conditions.
  Since $\Hat{c} = n - 2 \sum_{i} q_{i}$ can be considered as the dimension of a geometric counterpart,
  the (generalized) Calabi-Yau condition is presumed to be a necessary condition for Calabi-Yau/Landau-Ginzburg correspondence.
  In particular, we will consider the case of $\Hat{c} = 1, 2, 3$ here.
  One of the simplest example of $\Hat{c} = 3$ is a family of nondegenerate quintic homogeneous polynomials with five variables.
  We will calculate all possible Hodge numbers of the Landau-Ginzburg orbifold of a pair $(W, G)$ where $G$ satisfies
  $\langle J \rangle \subset G \subset \SL_{5}(\C)$, as a result of which we have the main theorem of the paper.

  \subsection*{Case of $\Hat{c} = 1$}
    We start with the case of $\Hat{c} = 1$ and the Calabi-Yau condition is satisfied $\sum_{i} q_{i} = 1$.
    Then the number of variables must be $3$.
    Let $W = x_{1}^{3} + x_{2}^{3} + x_{3}^{3}$ and $J = \diag(\zeta_{3}, \zeta_{3}, \zeta_{3})$.
    We consider the Poincaré polynomials for some $G \subset G_{W}$.

    At first, let $G = \langle J \rangle$.
    $G$ is an abelian group of the order $3$.
    We can easily see
    \begin{equation*}
      h^{p, q}(W, G) =
      \begin{cases}
        1 & \text{if $(p, q) = (0, 0), (0, 1), (1, 0), (1, 1)$,} \\
        0 & \text{otherwise.} \\
      \end{cases}
    \end{equation*}
    Observe that these Hodge numbers completely coincide with the ones of a $1$-dimensional Calabi-Yau manifold.
    We obtain the same Hodge numbers for all $G$ which is a diagonal symmetry group and satisfy $J \subset G \subset \SL_{3}(\C)$.

    Next, let us discuss an example which $G$ contains nondiagonal matrices.
    Since $W$ is invariant under a permutation of variables $x_{1} \to x_{2} \to x_{3} \to x_{1}$,
    we can set $G$ the group generated by $J$ and the following matrix:
    \begin{equation*} {\renewcommand\arraystretch{0.8}\setlength\arraycolsep{2pt}
      \begin{pmatrix}
        0 & 1 & 0 \\
        0 & 0 & 1 \\
        1 & 0 & 0
      \end{pmatrix}.
    } \end{equation*}
    Note that $G$ satisfies $\langle J \rangle \subset G \subset \SL_{3}(\C)$.
    This example gives the same Hodge numbers with the above examples.

    If $G$ does not satisfy $\langle G \rangle \subset G \subset \SL_{3}(\C)$, the situation is different.
    For example, let $G$ be a group generated by $\diag(\zeta_{3}, 1 ,1)$.
    Then the Poincaré polynomial is
    \begin{equation*}
      P(W, G; u, v) = u v^{2/3} + u^{2/3} v + 2 u^{2/3} v^{1/3} + 2 u^{1/3} v^{2} + u^{1/3} + v^{1/3}.
    \end{equation*}
    In general, the condition $\langle J \rangle \subset G \subset \SL_{n}(\C)$ seems to be necessary
    to obtain the Hodge numbers which look like the ones of some manifold.

  \subsection*{Case of $\Hat{c} = 2$}
    The case of $\Hat{c} = 2$ is similar to that of $\Hat{c} = 1$ since all $2$-dimensional Calabi-Yau manifolds also have the same Hodge numbers.
    For example, we can check that each diagonal symmetry group $G$ of $W = x_{1}^{4} + x_{2}^{4} + x_{3}^{4} + x_{4}^{4}$
    satisfying $\langle J \rangle \subset G \subset \SL_{4}(\C)$ yields the
    Hodge numbers
    \begin{equation*}
      h^{p, q}(W, G) =
      \begin{cases}
        1 & \text{if $(p, q) = (0, 0), (0, 2), (2, 0), (2, 2)$,} \\
        20 & \text{if $(p, q) = (1, 1)$,} \\
        0 & \text{otherwise.} \\
      \end{cases}
    \end{equation*}
    The same Hodge numbers appear if $G$ is generated by $J$ and the following two matrices which correspond to generators of
    the alternating group of degree $4$:
    \begin{equation*} {\renewcommand\arraystretch{0.8}\setlength\arraycolsep{2pt}
      \begin{pmatrix}
        0 & 1 & 0 & 0  \\
        0 & 0 & 1 & 0  \\
        0 & 0 & 0 & 1  \\
        1 & 0 & 0 & 0
      \end{pmatrix}, \quad
      \begin{pmatrix}
        1 & 0 & 0 & 0  \\
        0 & 0 & 1 & 0  \\
        0 & 0 & 0 & 1  \\
        0 & 1 & 0 & 0
      \end{pmatrix}.
    } \end{equation*}
    This is an example of a Landau-Ginzburg orbifold with a nonabelian group.

  \subsection*{Case of $\Hat{c} = 3$}
    In the examples of $\Hat{c} = 1, 2$, we obtained the same Hodge numbers regardless of $G$ if $\langle J \rangle \subset G \subset \SL_{n}(\C)$.
    We can guess that the case of $\Hat{c} = 3$ is somewhat different since a $3$-dimensional Calabi-Yau manifold $X$ does not have fixed values
    for $h^{1, 1}(X)$ and $h^{2, 1}(X)$.

    Let $W$ be a nondegenerate quintic homogeneous polynomial with five variables and $G$ be a finite subgroup of the symmetry group $G_{W}$.
    Put $J = \diag(\zeta_5, \dots, \zeta_5)$ and assume $\langle J \rangle \subset G \subset \SL_5(\C)$.
    $W$ defines a nonsingular quintic threefold $X = \{W=0\} \subset \P^4$.
    Let $\pi \colon \SL_5(\C) \to \PGL_5(\C)$ be a quotient map.
    $\Tilde{G} \defeq \pi(G)$ is a finite subgroup of $\Aut(X)$ which fixes a nowhere vanishing holomorphic $3$-form on $X$.
    From Theorem~\ref{thm:nonsingular-quintic-threefold-hodge-number}, we can examine the structure of $\Tilde{G}$ and hence $G$.

    \begin{theorem}
      \label{thm:main-theorem}
      In the above settings, $G$ is one of the groups listed in Table~\ref{table:quintic-landau-ginzburg-orbifold} of Appendix~\ref{sec:table}
      up to a linear change of the coordinates.
      In particular, we have the relation
      \begin{equation*}
        h^{p,q}(W, G) = h^{3-p, q}(X, \Tilde{G}) \quad \text{for all $p, q$}.
      \end{equation*}
    \end{theorem}

    \begin{proof}
      By Theorem~\ref{thm:nonsingular-quintic-threefold-hodge-number}, $\Tilde{G}$ is one of the groups listed in Table~\ref{table:quintic-threefold-orbifold}
      of Appendix~\ref{sec:table}.
      Since there is one-to-one correspondence between subgroups of the $\PGL_5(\C) \cong \PSL_5(\C) \cong \SL_5(\C)/\langle J \rangle$ and
      subgroups of $\SL_5(\C)$ which contain $\langle J \rangle$, $G$ is uniquely determined by $\Tilde{G}$.
      Therefore $G$ is one of the groups shown in Table~\ref{table:quintic-landau-ginzburg-orbifold} up to a linear change of the coordinates.
      Now we can calculate the Hodge numbers $h^{p,q}(W, G)$ using the formula in Theorem~\ref{thm:generalized-vafas-formula}.
      By direct calculations (see Appendix~\ref{sec:hodge-number-calculation}), we have
      \begin{align*}
        P(W, G; u, v) = &\ u^3v^3 + h^{1,1}(W, G)u^2v^2 \\ \tag{$\star$}
          & + u^3 + h^{2,1}(W, G)u^2v + h^{2,1}(W, G)u v^2 + v^3 \\
          & + h^{1,1}(W, G)u v + 1
      \end{align*}
      where $h^{1,1}(W, G)$ and $h^{2,1}(W, G)$ are the values listed in Table~\ref{table:quintic-landau-ginzburg-orbifold}.
      Comparing Table~\ref{table:quintic-threefold-orbifold} and Table~\ref{table:quintic-landau-ginzburg-orbifold},
      we see that the relation $h^{p,q}(W, G) = h^{3-p,q}(X, \Tilde{G})$ holds.
    \end{proof}
    As we can see from Section~\ref{sec:orbifold-quintic-threefold}, there is a $3$-dimensional Calabi-Yau manifold $Z$
    such that its Hodge numbers satisfy the relation $h^{p, q}(W, G) = h^{3 - p, q}(Z)$.

    Let us see one example in which $W$ is not homogeneous, for example $W = x_{1}^{8} + x_{2}^{8} + x_{3}^{4} + x_{4}^{4} + x_{5}^{4}$.
    We summarize the result of calculations of the Hodge numbers for some $G$.
    Since all calculation yield the Poincaré polynomial of the form ($\star$),
    it suffices to show values of $h^{1, 1}(W, G)$ and $h^{2, 1}(W, G)$.

    \begin{enumerate}
      \item Let $G$ be a group generated by $J$. Then we have
        \begin{equation*}
          h^{1, 1}(W, G) = 86, \quad h^{2, 1}(W, G) = 2.
        \end{equation*}
      \item Let $G$ be a group generated by $J$, $\diag(\zeta_{8}, \zeta_{8}^{-1}, 1, 1, 1)$, $\diag(\zeta_{4}, 1, \zeta_{4}^{-1}, 1, 1)$,
        $\diag(\zeta_{4}, 1, 1, \zeta_{4}^{-1}, 1)$ and $\diag(\zeta_{4}, 1, 1, \zeta_{4}^{-1})$.
        Then we have
        \begin{equation*}
          h^{1, 1}(W, G) = 2, \quad h^{2, 1}(W, G) = 86.
        \end{equation*}
      \item Let $G$ be a group generated by $J$ and the following matrix:
        \begin{equation*} {\renewcommand\arraystretch{0.8}\setlength\arraycolsep{2pt}
          \begin{pmatrix}
            1 & 0 & 0 & 0 & 0 \\
            0 & 1 & 0 & 0 & 0 \\
            0 & 0 & 0 & 1 & 0 \\
            0 & 0 & 0 & 0 & 1 \\
            0 & 0 & 1 & 0 & 0 \\
          \end{pmatrix}.
        } \end{equation*}
        Then we have
        \begin{equation*}
          h^{1, 1}(W, G) = 52, \quad h^{2, 1}(W, G) = 8.
        \end{equation*}
      \item Let $G$ be a group generated by $J$ and the following matrix:
        \begin{equation*} {\renewcommand\arraystretch{0.8}\setlength\arraycolsep{2pt}
          \begin{pmatrix}
            0 & -1 & 0 & 0 & 0 \\
            -1 & 0 & 0 & 0 & 0 \\
            0 & 0 & -1 & 0 & 0 \\
            0 & 0 & 0 & -1 & 0 \\
            0 & 0 & 0 & 0 & -1 \\
          \end{pmatrix}.
        } \end{equation*}
        Then we have
        \begin{equation*}
          h^{1, 1}(W, G) = 52, \quad h^{2, 1}(W, G) = 4.
        \end{equation*}
    \end{enumerate}

    As our final example, consider $W = x_{1}^{3} + x_{2}^{3} + x_{3}^{3} + x_{4}^{3} + x_{5}^{3} + x_{6}^{3} + x_{7}^{3} + x_{8}^{3} + x_{9}^{3}$.
    Notice that $W$ satisfies the generalized Calabi-Yau condition.
    Let $G$ be a group generated by $J = \diag(\zeta_{3}, \dots, \zeta_{3})$.
    The Poincaré polynomial of the pair $(W, G)$ is
    \begin{equation*}
      P(W, G; u, v) = u^3v^3 + 84 u^2v^2 + u^3 +  v^3 + 84 u v + 1,
    \end{equation*}
    in particular $h^{2, 1}(W, G) = 0$.
    There is no $3$-dimensional Calabi-Yau manifold $X$ s.t. the relation $h^{p, q}(W, G) = h^{3 - p, q}(X)$.
    Indeed, if $X$ exists, $h^{1,1}(X) = 0$.
    However this means that the Kähler form of $X$ is absent, therefore $X$ cannot be Kähler.

\section{Conclusion}
  \label{sec:conclusion}
  In this paper, we have defined the Landau-Ginzburg orbifold of a nondegenerate quasihomogeneous polynomial
  and its finite symmetry group as a bigraded vector space.
  We calculated the Hodge numbers of the Landau-Ginzburg orbifold of nondegenerate quintic homogeneous polynomials
  with five variables using the formula of the Poincaré polynomial, and proved these Hodge numbers
  have a certain relation with the Hodge numbers of the corresponding geometric counterparts.

  In the following, we comment on some future prospects related to the Landau-Ginzburg orbifold
  which we discussed.

  \subsection*{The Witten index, E-polynomial, and $\Z_{2}$-grading}
    As we have already seen, the Witten index was defined when $h^{p, q}(W, G) \neq 0$ implies $p - q \in \Z$.
    Furthermore, we can define a more generic notion, the \emph{E-polynomial} $E(W, G; u, v)$, by
    \begin{equation*}
      E(W, G; u, v) = P(W, G; -u, -v) = \sum_{p, q \in \Q} (-1)^{p - q} h^{p, q}(W, G) u^{p} v^{q}.
    \end{equation*}
    This is an analogous notion to the E-polynomial for algebraic varieties defined by means of the mixed
    Hodge structure of rational cohomology groups with compact supports~\cite{batyrev1996}.
    The E-polynomial of varieties $X$ has the addition property
    \begin{equation*}
      E(X; u, v) = \sum_{i} E(X_{i}; u, v)
    \end{equation*}
    where $X$ is a disjoint union of locally closed subvarieties $X_{i}$, and the multiplication property
    \begin{equation*}
      E(X; u, v) = E(Y; u, v) E(F; u, v)
    \end{equation*}
    where $\pi \colon X \to Y$ is a locally trivial fibration in Zariski topology and $F$ is the fiber over
    a closed point in $Y$.
    It is natural to expect similar properties for the E-polynomial of Landau-Ginzburg orbifolds.
    For this idea, the notion corresponding to ``direct sum'' or ``fibration'' may also to be defined
    for Landau-Ginzburg orbifolds.

    There is another approach to these concepts.
    Recently, Basalaev, Takahashi and Werner~\cite{takahashi2016} proposed an axiom of the Landau-Ginzburg B-model
    for diagonal symmetry groups, which is equipped with $\Z_{2}$-grading
    $\mathscr{H}_{W, G} = \mathscr{H}_{W, G, \bar{0}} \oplus \mathscr{H}_{W, G, \bar{1}}$
    defined so that $n - n_{g}$ determines the parity of a $g$-twisted sector.
    This $\Z_{2}$-grading should correspond to the bosonic and fermionic state space of a conformal field theory,
    hence we may be able to define the E-polynomial by
    \begin{equation*}
      E(W, G; u, v) = \sum_{p, q \in \Q} \Big( \sum_{g \in S} (-1)^{n - n_{g}} h_{g}^{p, q}(W, G) \Big) u^{p} v^{q}.
    \end{equation*}
    One of the advantages using this definition is that we can define the Witten index by $E(W, G; 1, 1)$
    even if there is a pair $(p, q)$ s.t. $p - q \not\in \Z$ and $h^{p, q}(W, G) \neq 0$.
    Note that if $h^{p, q}(W, G) \neq 0$ implies $p - q \in \Z$ and $G$ is a subgroup of $\SL_{n}(\C)$,
    then two definitions of E-polynomial coinside since we have
    \begin{equation*}
      (-1)^{p - q} = (-1)^{\age(g) - \age(g^{-1})} = (-1)^{\age(g) + \age(g^{-1})} = (-1)^{n - n_{g}}
    \end{equation*}
     from Lemma~\ref{lem:age-ng-relation}.

  \subsection*{Frobenius algebra structure}
    Our definition of the Landau-Ginzburg B-model does not have the product structure.
    We should endow the Landau-Ginzburg orbifold with a product and Frobenius structure induced from
    the operator algebra of the conformal field theory.
    Krawitz's work~\cite{krawitz2010}, which defined the product and Frobenius structure
    for the pair of an invertible polynomial and a diagonal symmetric group,
    will be a signpost to accomplish this goal.
    If we have the Frobenius structure for the Landau-Ginzburg B-model, it will be a first step for
    the search of the proof of the Landau-Ginzburg mirror symmetry.

  \subsection*{Calabi-Yau/Landau-Ginzburg correspondence}
    The main theorem of the paper says that the Landau-Ginzburg orbifold of a nondegenerate quintic homogeneous
    polynomial $W$ with five variables and a finite subgroup $G$ of the symmetry group $G_{W}$,
    s.t. $\langle J \rangle \subset G \subset \SL_{n}(\C)$ has the geometric counterpart.
    What sort of pair of polynomial and its symmetry group has the geometric counterpart?
    Since the central charge $\Hat{c}$ of the Landau-Ginzburg orbifold corresponds to the dimension of
    the counterpart object, the condition $\Hat{c} \in \Z_{\geq 0}$ seems to be a necessary condition.
    This is satisfied if a polynomial has the generalized Calabi-Yau condition.
    In addition, there might be some conditions on the symmetry group $G$.
    For example, the condition $\langle J \rangle \subset G \subset \SL_{n}(\C)$ seems to ensure that
    $h^{p, q}(W, G) \neq 0$ implies $p,q \in \Z_{\geq 0}$ if the Calabi-Yau condition is satisfied.
    However, this seems to be not sufficient in general.

\appendix
\section{Two tables of the Hodge numbers}
  \label{sec:table}
  This section contains Table~\ref{table:quintic-threefold-orbifold} and
  Table~\ref{table:quintic-landau-ginzburg-orbifold}.
  Throughout this section, $\zeta_{n}$ means the primitive $n$-th root of unity and $J$ is the exponential grading operator
  for a quintic homogeneous polynomial with five variables, i.e.
  \begin{equation*}
    J = \diag(\zeta_{5}, \zeta_{5}, \zeta_{5}, \zeta_{5}, \zeta_{5}).
  \end{equation*}

  Table~\ref{table:quintic-threefold-orbifold}, originally given by Yu~\cite{Yu_2016}, shows the orbifold Hodge
  numbers and the orbifold Euler number
  for the pair $(X, \Tilde{G})$ where
  $X$ is a nonsingular quintic threefold and $\Tilde{G}$ is a subgroup of $\Aut(X)$ which fixes
  a nowhere vanishing holomorphic top form on $X$.
  Similarly, Table~\ref{table:quintic-landau-ginzburg-orbifold} gives the Hodge numbers and the Witten index
  for the Landau-Ginzburg orbifold of the pair $(W, G)$ where $W$ is a nondegenerate quintic homogeneous polynomial and
  $G$ is a finite subgroup of $G_{W}$ which satisfy $\langle J \rangle \subset G \subset \SL_{5}(\C)$.
  See Appendix~\ref{sec:hodge-number-calculation} for the detailed calculation algorithm of these Hodge numbers.

  We should explain notation in Table~\ref{table:quintic-threefold-orbifold} in more detail.
  The second column shows generators of $\Tilde{G}$ as subgroups of $\PGL_{5}(\C)$.
  We denote the matrices whose action are permutations of $x_{i}$ by elements of the symmetric group $S_{5}$.
  For example, $(12)(34)$ means
  \begin{equation*} {\renewcommand\arraystretch{0.8}\setlength\arraycolsep{2pt}
    \begin{pmatrix}
      0 & 1 & 0 & 0 & 0 \\
      1 & 0 & 0 & 0 & 0 \\
      0 & 0 & 0 & 1 & 0 \\
      0 & 0 & 1 & 0 & 0 \\
      0 & 0 & 0 & 0 & 1
    \end{pmatrix}.
  } \end{equation*}
  In addition, we use following matrices:
  {\renewcommand\arraystretch{0.9}\setlength\arraycolsep{2pt}
  \begin{align*}
    \mathscr{A}_{1} &= \begin{pmatrix}
      \zeta_{8}^{3} & 0 & 0 & 0 & 0 \\
      0 & \zeta_{8} & 0 & 0 & 0 \\
      0 & 0 & 0 & 1 & 0 \\
      0 & 0 & 1 & 0 & 0 \\
      0 & 0 & 0 & 0 & 1
    \end{pmatrix}, &
    \mathscr{A}_{2} &= \begin{pmatrix}
      0 & 1 & 0 & 0 & 0 \\
      -1 & 0 & 0 & 0 & 0 \\
      0 & 0 & 1 & 0 & 0 \\
      0 & 0 & 0 & 1 & 0 \\
      0 & 0 & 0 & 0 & 1
    \end{pmatrix}, &
    \mathscr{A}_{3} &= \begin{pmatrix}
      -\frac{1}{\sqrt{2}}\zeta_{8} & \frac{1}{\sqrt{2}}\zeta_{8} & 0 & 0 & 0 \\
      \frac{1}{\sqrt{2}}\zeta_{8}^{3} & \frac{1}{\sqrt{2}}\zeta_{8}^{3} & 0 & 0 & 0 \\
      0 & 0 & \zeta_{3} & 0 & 0 \\
      0 & 0 & 0 & \zeta_{3}^{2} & 0 \\
      0 & 0 & 0 & 0 & 1
    \end{pmatrix}, \\
    \mathscr{A}_{4} &= \begin{pmatrix}
      0 & 1 & 0 & 0 & 0 \\
      0 & 0 & 1 & 0 & 0 \\
      1 & 0 & 0 & 0 & 0 \\
      0 & 0 & 0 & \zeta_{3} & 0 \\
      0 & 0 & 0 & 0 & \zeta_{3}^{2}
    \end{pmatrix}, &
    \mathscr{A}_{5} &= \begin{pmatrix}
      \zeta_{3}^{2} & 0 & 0 & 0 & 0 \\
      0 & \zeta_{3} & 0 & 0 & 0 \\
      0 & 0 & 0 & 1 & 0 \\
      0 & 0 & 0 & 0 & 1 \\
      0 & 0 & 1 & 0 & 0
    \end{pmatrix}. & \\
  \end{align*}
  }
  The third column, labeled as $X$, shows examples of quintic threefolds which admit the $\Tilde{G}$-action which fixes a nowhere vanishing $3$-form on $X$.
  $X_1, \dots, X_{14}$ are defined as the set of zeros of the following homogeneous polynomials respectively:
  \begin{align*}
    W_{1} &= x_{1}^{5} + x_{2}^{5} + x_{3}^{5} + x_{4}^{5} + x_{5}^{5}, \\
    W_{2} &= x_{1}^{4}x_{2} + x_{2}^{4}x_{1} + x_{3}^{4} x_{4} + x_{4}^{4} x_{3} + x_{5}^{5}, \\
    W_{3} &= x_{1}^{4}x_{2} + x_{2}^5 + x_{3}^{4}x_{4} + x_{4}^5 + x_{5}^5, \\
    W_{4} &= x_{1}^{4}x_{4} + x_{2}^{4}x_{5} + x_{3}^{4}x_{4} + x_{4}^{5} + x_{5}^{5} + x_{1}x_{2}x_{3}^{3}, \\
    W_{5} &= ((x_{1}^4 + x_{2}^4) + (2 + 4\zeta_{3}^{2})x_{1}^{2}x_{2}^{2})x_{3}+(-(x_{1}^4+x_{2}^4)
             + (2 + 4\zeta_{3}^2)x_{1}^{2}x_{2}^2)x_{4} \\
          &\qquad + x_{3}^{4}x_{4} + x_{4}^{4}x_{3} + x_{5}^{5}, \\
    W_{6} &= x_{1}^{4}x_{2} + x_{2}^{4}x_{3} + x_{3}^{5} + x_{4}^{4}x_{2} + x_{5}^{4}x_{3} + x_{1}x_{4}x_{5}^{3} + x_{2}^{3}x_{5}^{2} + x_{1}x_{3}x_{4}^{3} + x_{1}^{3}x_{3}x_{4}, \\
    W_{7} &= x_{1}^{4}x_{2} + x_{2}^4x_{1} + x_{3}^{4}x_{2} + x_{4}^{4}x_{1} + x_{5}^{5} + x_{1}^{3}x_{3}x_{4} + x_{2}^{3}x_{3}x_{4}, \\
    W_{8} &= ((x_{1}^4 + x_{2}^4) + (2 + 4\zeta_{3}^{2})x_{1}^{2}x_{2}^{2})x_{3}+((x_{1}^4+x_{2}^4)
             - (2 + 4\zeta_{3}^2)x_{1}^{2}x_{2}^2)x_{4} + x_{3}^{4}x_{4} \\
          &\qquad + x_{4}^{4}x_{3} + x_{3}^2x_{4}^2x_{5} + x_{5}^{5}, \\
    W_{9} &= x_{1}^{4}x_{2} + x_{2}^{4}x_{1} + x_{3}^{4}x_{2} + x_{4}^{4}x_{1} + x_{5}^{5} + x_{2}^{3}x_{3}x_{4} - x_{1}^3x_{3}x_{4}, \\
    W_{10} &= x_{1}^{4}x_{2} + x_{2}^{4}x_{3} + x_{3}^{4}x_{1} + x_{4}^{4}x_{5} + x_{5}^{4}x_{4}, \\
    W_{11} &= x_{1}^{4}x_{2} + x_{2}^{4}x_{1} + x_{3}^{5} + x_{4}^{5} + x_{5}^{5}, \\
    W_{12} &= x_{1}^{4}x_{2} + x_{2}^{4}x_{3} + x_{3}^{4}x_{4} + x_{4}^{4}x_{1} + x_{5}^{5}, \\
    W_{13} &= x_{1}^{4}x_{2} + x_{2}^{4}x_{3} + x_{3}^{4}x_{4} + x_{4}^{4}x_{5} + x_{5}^{4}x_{1}, \\
    W_{14} &= x_{1}^{4}x_{2} + x_{2}^{4}x_{3} + x_{3}^{4}x_{1} + x_{4}^{5} + x_{5}^{5}.
  \end{align*}

  {\small\renewcommand\arraystretch{1.5}
    \begin{longtable}{p{\dimexpr 0.06\linewidth-2\tabcolsep}
                  >{\raggedright\let\newline\\\arraybackslash\hspace{0pt}}p{\dimexpr 0.60\linewidth-2\tabcolsep}
                  p{\dimexpr 0.08\linewidth-2\tabcolsep}
                  p{\dimexpr 0.10\linewidth-2\tabcolsep}
                  p{\dimexpr 0.16\linewidth-2\tabcolsep}}
  \caption{The orbifold Hodge numbers and orbifold Euler number of $X/\Tilde{G}$}
  \label{table:quintic-threefold-orbifold}
  \\ \hline
  \# & Generators of $\Tilde{G}$ (as a subgroup of $\PGL_{5}(\C)$) & $X$ & $e(X, \Tilde{G})$ & $(h^{1,1},\;h^{2,1})$ \\ \hline
  \endfirsthead
  \caption[]{\textit{(Continued from previous page)}} \\ \hline
  \# & Generators of $\Tilde{G}$ (as a subgroup of $\PGL_{5}(\C))$ & $X$  & $e(X, \Tilde{G})$ &$(h^{1,1},\;h^{2,1})$ \\ \hline
  \endhead
  \hline \multicolumn{4}{r}{\textit{(Continued on next page)}} \\
  \endfoot
  \hline
  \endlastfoot

  1 & & Any & $-200$ & $(1,101)$ \\
  2 & $(12)(34)$ & $X_{1}$   &$-112$  & $(3,59)$ \\
  3 & $(123)$ & $X_{1}$   &$-88$  & $(5,49)$ \\
  4 & $\diag(\zeta_3,\zeta_3^2,\zeta_3,\zeta_3^2,1)$ &$X_{2}$   &$-56$  & $(5,33)$ \\
  5 & $\diag(\zeta_4,1,\zeta_4^3,1,1)$ &$X_{3}$   &$-80$  & $(7,47)$ \\
  6 & $\diag(\zeta_4,\zeta_4,-1,1,1)$ &$X_{4}$   &$-32$  & $(11,27)$ \\
  7 & $(12)(34)$, $(13)(24)$ &$X_{1}$   &$-56$  & $(7,35)$ \\
  8 & $\diag(1,\zeta_5,\zeta_5^4,1,1)$ &$X_{1}$   &$-88$  & $(5,49)$ \\
  9 & $\diag(\zeta_5,\zeta_5,\zeta_5^4,\zeta_5^4,1)$ &$X_{1}$   &$8$  & $(21,17)$ \\
  10 & $\diag(1,\zeta_5,\zeta_5^2,\zeta_5^3,\zeta_5^4)$ &$X_{1}$   &$-40$  & $(1,21)$ \\
  11 & $\diag(\zeta_3,\zeta_3^2,1,1,1)$, $(12)(34)$ &$X_{2}$   &$-56$  & $(5,33)$ \\
  12 & $\diag(\zeta_3,\zeta_3^2,\zeta_3,\zeta_3^2,1)$,$(12)(34)$ &$X_{2}$   &$-40$  & $(5,25)$ \\
  13 & $\diag(\zeta_3,\zeta_3^2,\zeta_3,\zeta_3^2,1)$,$(13)(24)$ &$X_{2}$   &$-16$  & $(11,19)$ \\
  14 & $\mathscr{A}_{1}$&$X_{5}$   &$8$  & $(23,19)$ \\
  15 & $\diag(\zeta_8,-1,1,\zeta_8^5,\zeta_4^3)$ &$X_{6}$   &$8$  & $(17,13)$ \\
  16 & $\diag(\zeta_4,\zeta_4,-1,1,1)$, $\diag(1,-1,-1,1,1)$ &$X_{4}$   &$-4$  & $(19,21)$ \\
  17 & $(12)(34)$, $\diag(1,1,\zeta_4,\zeta_4^3,1)$ &$X_{7}$   &$-40$  & $(9,29)$ \\
  18 & $\mathscr{A}_{2}$, $\diag(\zeta_4^3,\zeta_4,1,1,1)$ &$X_{8}$   &$-64$  & $(9,41)$ \\
  19 & $\diag(\zeta_3,\zeta_3^2,1,1,1)$, $\diag(1,1,\zeta_3,\zeta_3^2,1)$ &$X_{2}$   &$-8$  & $(17,21)$ \\
  20 & $(12)(34)$,$\diag(\zeta_5,\zeta_5^4,\zeta_5,\zeta_5^4,1)$ &$X_{1}$   &$-8$  & $(13,17)$ \\
  21 & $(12)(34)$, $\diag(\zeta_5,\zeta_5^4,1,1,1)$ &$X_{1}$   &$-56$  & $(5,33)$ \\
  22 & $(12)(34)$,$\diag(\zeta_5,\zeta_5^4,\zeta_5^2,\zeta_5^3,1)$ &$X_{1}$   &$-32$  & $(3,19)$ \\
  23 & $\diag(-1,1,-\zeta_5,\zeta_5,\zeta_5^3)$ &$X_{3}$   &$16$  & $(19,11)$ \\
  24 & $\diag(\zeta_3,\zeta_3^2,\zeta_3,\zeta_3^2,1)$, $\diag(1,1,\zeta_4,\zeta_4^3,1)$ &$X_{9}$   &$16$  & $(23,15)$ \\
  25 & $(12)(34)$, $(13)(24)$, $(123)$ &$X_{1}$   &$-40$  & $(7,27)$ \\
  26 & $\diag(\zeta_3,\zeta_3^2,\zeta_3,\zeta_3^2,1)$,$(13)(24)$,$(12)(34)$ &$X_{2}$   &$-8$  & $(11,15)$ \\
  27 & $\diag(\zeta_{13},\zeta_{13}^9,\zeta_{13}^3,1,1)$ &$X_{10}$   &$88$  & $(49,5)$ \\
  28 & $\diag(\zeta_3,\zeta_3^2,1,1,1)$, $\diag(1,1,\zeta_5,\zeta_5,\zeta_5^3)$ &$X_{2}$   &$88$  & $(49,5)$ \\
  29 & $\diag(\zeta_3,\zeta_3^2,1,1,1)$, $\diag(1,1,\zeta_5,1,\zeta_5^4)$ &$X_{11}$   &$-8$  & $(17,21)$ \\
  30 & $\diag(\zeta_3,\zeta_3^2,\zeta_3,\zeta_3^2,1)$,$\diag(1,1,\zeta_5,\zeta_5,\zeta_5^3)$ &$X_{2}$   &$56$  & $(33,5)$ \\
  31 & $\mathscr{A}_{1}$, $\mathscr{A}_{2}$ &$X_{5}$   &$-8$  & $(17,21)$ \\
  32 & $\diag(\zeta_{17},\zeta_{17}^{-4},\zeta_{17}^{16},\zeta_{17}^4,1)$ &$X_{12}$   &$56$  & $(33,5)$ \\
  33 & $\diag(\zeta_3,\zeta_3^2,1,1,1)$,$\diag(1,1,\zeta_3,\zeta_3^2,1)$, $(13)(24)$ &$X_{2}$   &$8$  & $(17,13)$ \\
  34 & $\diag(\zeta_3,\zeta_3^2,1,1,1)$,$\diag(1,1,\zeta_3,\zeta_3^2,1)$, $(12)(34)$ &$X_{2}$   &$-16$  & $(11,19)$ \\
  35 & $\diag(\zeta_4,1,\zeta_4^3,1,1)$, $\diag(1,1,\zeta_5,\zeta_5,\zeta_5^3)$ &$X_{3}$   &$32$  & $(27,11)$ \\
  36 & $\diag(-\zeta_5,\zeta_5,-\zeta_5^4,\zeta_5^4,1)$,$(13)(24)$ &$X_{3}$   &$8$  & $(15,11)$ \\
  37 & $\mathscr{A}_{2}$, $\mathscr{A}_{3}$&$X_{5}$   &$0$  & $(13,13)$ \\
  38 & $\diag(\zeta_3,\zeta_3^2,\zeta_3,\zeta_3^2,1)$,$\diag(1,1,\zeta_4,\zeta_4^3,1)$, $(12)(34)$ &$X_{7}$   &$8$  & $(17,13)$ \\
  39 & $\diag(\zeta_5,\zeta_5^4,1,1,1)$, $\diag(1,1,\zeta_5,\zeta_5^4,1)$ &$X_{1}$   &$-8$  & $(17,21)$ \\
  40 & $\diag(\zeta_5,\zeta_5^4,1,1,1)$, $\diag(1,1,\zeta_5,\zeta_5,\zeta_5^3)$ &$X_{1}$   &$88$  & $(49,5)$ \\
  41 & $\diag(1,\zeta_5,\zeta_5,\zeta_5^4,\zeta_5^4)$, $\diag(1,1,\zeta_5,\zeta_5,\zeta_5^3)$ &$X_{1}$   &$40$  & $(21,1)$ \\
  42 & $\diag(1,\zeta_5,\zeta_5^2,\zeta_5^3,\zeta_5^4)$,$(12345)$ &$X_{1}$   &$-8$  & $(1,5)$ \\
  43 & $\diag(\zeta_3,\zeta_3^2,\zeta_5,\zeta_5,\zeta_5^3)$,$(12)(34)$ &$X_{2}$   &$56$  & $(33,5)$ \\
  44 & $\diag(\zeta_3,\zeta_3^2,\zeta_{15}^8,\zeta_{15}^{13},\zeta_5^3)$,  $(12)(34)$ &$X_{2}$   &$40$  & $(25,5)$ \\
  45 & $\diag(\zeta_{15}^{11},\zeta_{15}^1,\zeta_{15}^{14},\zeta_{15}^{4},1)$, $(13)(24)$ &$X_{2}$   &$40$  & $(25,5)$ \\
  46 & $\diag(\zeta_3,\zeta_3^2,\zeta_5,\zeta_5^4,1)$,  $(12)(34)$ &$X_{11}$   &$-16$  & $(11,19)$ \\
  47 & $\diag(\zeta_{15}^{11},\zeta_{15}^1,\zeta_{15}^{14},\zeta_{15}^{4},1)$, $(14)(23)$ &$X_{2}$   &$16$  & $(19,11)$ \\
  48 & $\diag(\zeta_{17},\zeta_{17}^{-4},\zeta_{17}^{-1},\zeta_{17}^{4},1)$, $(13)(24)$ &$X_{12}$   &$16$  & $(19,11)$ \\
  49 & $\diag(\zeta_{3},\zeta_{3}^2,1,1,1)$,$\diag(1,1,\zeta_{3},\zeta_{3}^{2},1)$, $(14)(23)$,$(13)(24)$ &$X_{2}$   &$16$  & $(17,9)$ \\
  50 & $\diag(\zeta_{13},\zeta_{13}^9,\zeta_{13}^{3},1,1)$, $(123)$ &$X_{10}$   &$8$  & $(21,17)$ \\
  51 & $\diag(\zeta_{13},\zeta_{13}^9,\zeta_{13}^{3},1,1)$, $\mathscr{A}_{4}$ &$X_{10}$   &$40$  & $(21,1)$ \\
  52 & $\diag(\zeta_{13},\zeta_{13}^9,\zeta_{13}^{3},1,1)$, $\diag(1,1,1,\zeta_3,\zeta_3^2)$ &$X_{10}$   &$200$  & $(101,1)$ \\
  53 & $\diag(\zeta_{20}^9,\zeta_{5},\zeta_{20}^{11},\zeta_{5}^{4},1)$, $(13)(24)$ &$X_{3}$   &$16$  & $(19,11)$ \\
  54 & $\diag(\zeta_{41},\zeta_{41}^{-4},\zeta_{41}^{16},\zeta_{41}^{18},\zeta_{41}^{10})$ &$X_{13}$   &$200$  & $(101,1)$ \\
  55 & $\diag(\zeta_{3},\zeta_{3}^2,\zeta_5,\zeta_5,\zeta_5^3)$, $\diag(1,1,\zeta_{3},\zeta_{3}^2,1)$ &$X_{2}$   &$200$  & $(101,1)$ \\
  56 & $\mathscr{A}_{1}$, $\mathscr{A}_{2}$, $\mathscr{A}_{3}$ &$X_{5}$   &$24$  & $(19,7)$ \\
  57 & $\diag(1,1,\zeta_{5},\zeta_{5},\zeta_{5}^{3})$,$\diag(\zeta_{5},\zeta_{5}^4,\zeta_{5},\zeta_{5}^{4},1)$, $(12)(34)$ &$X_{1}$   &$8$  & $(17,13)$ \\
  58 & $\diag(1,1,\zeta_{5},\zeta_{5},\zeta_{5}^{3})$, $\diag(\zeta_{5},\zeta_{5}^4,1,1,1)$, $(12)(34)$ &$X_{1}$   &$56$  & $(33,5)$ \\
  59 & $\diag(1,1,\zeta_{5},\zeta_{5},\zeta_{5}^{3})$,$\diag(\zeta_{5},\zeta_{5}^4,\zeta_{5}^2,\zeta_{5}^{3},1)$, $(12)(34)$ &$X_{1}$   &$32$  & $(19,3)$ \\
  60 & $\diag(1,1,\zeta_{5},\zeta_{5}^4,1)$, $\diag(\zeta_{5},\zeta_{5}^4,1,1,1)$, $(12)(34)$ &$X_{1}$   &$-16$  & $(11,19)$ \\
  61 & $\diag(\zeta_{51}^{20},\zeta_{51}^{22},\zeta_{51}^{14},\zeta_{51}^{46},1)$ &$X_{12}$   &$200$  & $(101,1)$ \\
  62 & $(12345)$,  $(123)$ &$X_{1}$   &$-16$  & $(5,13)$ \\
  63 & $\diag(\zeta_{15}^8,\zeta_{15}^{13},\zeta_{15}^2,\zeta_{15}^{7},1)$, $(12)(34)$, $(13)(24)$ &$X_{2}$   &$32$  & $(21,5)$ \\
  64 & $\diag(1,1,1,\zeta_{5},\zeta_{5}^{4})$, $\diag(\zeta_{13},\zeta_{13}^{-4},\zeta_{13}^3,1,1)$ &$X_{14}$   &$200$  & $(101,1)$ \\
  65 & $\diag(1,1,\zeta_{5},1,\zeta_{5}^{4})$, $(345)$ &$X_{1}$   &$8$  & $(21,17)$ \\
  66 & $\diag(1,1,\zeta_{5},1,\zeta_{5}^{4})$, $\mathscr{A}_{5}$ &$X_{11}$   &$40$  & $(21,1)$ \\
  67 & $\diag(1,1,\zeta_{5},1,\zeta_{5}^{4})$, $\diag(\zeta_3,\zeta_3^2,1,\zeta_5,\zeta_5^4)$ &$X_{11}$   &$200$  & $(101,1)$ \\
  68 & $\diag(\zeta_{15}^{8},\zeta_{15}^{13},\zeta_{5}^4,\zeta_{5}^{4},1)$,$\diag(1,1,\zeta_{3},\zeta_{3}^2,1)$, $(14)(23)$&$X_{2}$   &$112$  & $(59,3)$ \\
  69 & $\diag(\zeta_3,\zeta_3^2,\zeta_{5},\zeta_5, \zeta_{5}^{3})$,$\diag(1,1,\zeta_{3},\zeta_{3}^2,1)$, $(12)(34)$ &$X_{2}$   &$112$  & $(59,3)$ \\
  70 & $\diag(\zeta_5,\zeta_5,\zeta_{5}^4,\zeta_{5}^{4},1)$,$\diag(\zeta_5,\zeta_5^4,\zeta_{5},\zeta_{5}^{4},1)$, $(12)(34)$, $(13)(24)$ &$X_{1}$   &$16$  & $(17,9)$ \\
  71 & $\diag(\zeta_3,\zeta_3^2,\zeta_{3},\zeta_{3}^{2},1)$, $\diag(\zeta_{17},\zeta_{17}^{-4},\zeta_{17}^{16},\zeta_{17}^{4},1)$, $(13)(24)$ &$X_{12}$   &$112$  & $(59,3)$ \\
  72 & $\diag(\zeta_{13},\zeta_{13}^{-4},\zeta_{13}^3,\zeta_{3},\zeta_3^2)$, $(123)$ &$X_{10}$   &$88$  & $(49,5)$ \\
  73 & $\diag(1, \zeta_5,\zeta_5^4,\zeta_{5}^4,\zeta_{5})$,   $(12345)$ &$X_{1}$   &$8$  & $(5,1)$ \\
  74 & $\diag(1,\zeta_5,1,1,\zeta_{5}^4)$, $\diag(1,1,\zeta_5,1,\zeta_{5}^{4})$, $\diag(1,1,1,\zeta_5,\zeta_{5}^{4})$ &$X_{1}$   &$200$  & $(101,1)$ \\
  75 & $\diag(1,1, \zeta_5,1,\zeta_{5}^4)$, $\diag(1,1,1,\zeta_5,\zeta_{5}^{4})$,  $(345)$, $(12)(34)$ &$X_{1}$   &$16$  & $(19,11)$ \\
  76 & $\diag(1,1, \zeta_5,1,\zeta_{5}^4)$, $\diag(1,1,1,\zeta_5,\zeta_{5}^{4})$,  $\mathscr{A}_{5}$, $(12)(34)$ &$X_{11}$   &$32$  & $(19,3)$ \\
  77 & $\diag(1,1, \zeta_5,\zeta_5,\zeta_{5}^3)$,$\diag(\zeta_3,\zeta_3^2,\zeta_5,\zeta_{5}^{4},1)$, $(12)(34)$ &$X_{11}$   &$112$  & $(59,3)$ \\
  78 & $\diag(1,1, \zeta_3,\zeta_{3}^2,1)$, $\diag(\zeta_{15}^8,\zeta_{15}^{13},\zeta_5^4,\zeta_{5}^{4},1)$,  $(12)(34)$, $(13)(24)$ &$X_{2}$   &$80$  & $(41,1)$ \\
  79 & $\diag(\zeta_{13},\zeta_{13}^{-4}, \zeta_{13}^3,\zeta_5,\zeta_{5}^4)$,  $(123)$ &$X_{14}$   &$88$  & $(49,5)$ \\
  80 & $\diag(\zeta_{41},\zeta_{41}^{-4},\zeta_{41}^{16},\zeta_{41}^{18},\zeta_{41}^{10})$, $(12345)$ &$X_{13}$   &$40$  & $(21,1)$ \\
  81 & $\diag(\zeta_3,\zeta_3^2, \zeta_5,1,\zeta_5^4)$, $\diag(1,1,1,\zeta_{5},\zeta_5^4)$,  $(345)$ &$X_{11}$   &$88$  & $(49,5)$ \\
  82 & $\diag(1,\zeta_5, \zeta_5^4,\zeta_{5}^4,\zeta_5)$, $\diag(1,\zeta_{5},\zeta_5^2,\zeta_{5}^{3},\zeta_5^4)$,  $(12345)$, $(25)(34)$ &$X_{1}$   &$16$  & $(11,3)$ \\
  83 & $\diag(1,\zeta_5, \zeta_5,\zeta_{5}^4,\zeta_5^4)$, $\diag(1,\zeta_{5},\zeta_5^4,1,1)$, $\diag(1,1,1,\zeta_{5},\zeta_5^4)$,  $(23)(45)$ &$X_{1}$   &$112$  & $(59,3)$ \\
  84 & $\diag(\zeta_5, \zeta_5^4,1,1,1)$, $\diag(1,1,\zeta_{5},1,\zeta_5^4)$, $\diag(1,1,1,\zeta_{5},\zeta_5^4)$,  $(345)$ &$X_{1}$   &$88$  & $(49,5)$ \\
  85 & $\diag(\zeta_3,\zeta_3^2,\zeta_5,1,\zeta_{5}^4)$, $\diag(1,1,1,\zeta_{5},\zeta_5^4)$, $(345)$,  $(12)(34)$ &$X_{11}$   &$56$  & $(33,5)$ \\
  86 & $\diag(1,\zeta_5, \zeta_5,\zeta_{5}^4,\zeta_5^4)$, $\diag(1,\zeta_{5},\zeta_5^4,1,1)$, $\diag(1,1,1,\zeta_{5},\zeta_5^4)$,  $(23)(45)$, $(24)(35)$ &$X_{1}$   &$80$  & $(41,1)$ \\
  87 & $\diag(1,\zeta_5,1,1,\zeta_{5}^4)$, $\diag(1,1,\zeta_5,1,\zeta_{5}^{4})$, $\diag(1,1,1,\zeta_5,\zeta_{5}^{4})$, $(12345)$ &$X_{1}$   &$40$  & $(21,1)$ \\
  88 & $\diag(\zeta_5, \zeta_5^4,1,1,1)$, $\diag(1,1,\zeta_{5},1,\zeta_5^4)$, $\diag(1,1,1,\zeta_{5},\zeta_5^4)$,  $(345)$, $(12)(34)$ &$X_{1}$   &$56$  & $(33,5)$ \\
  89 & $\diag(1,\zeta_5,1,1,\zeta_{5}^4)$, $\diag(1,1,\zeta_5,1,\zeta_{5}^{4})$, $\diag(1,1,1,\zeta_5,\zeta_{5}^{4})$, $(12345)$, $(25)(34)$ &$X_{1}$   &$32$  & $(19,3)$ \\
  90 & $\diag(1,\zeta_5, \zeta_5,\zeta_{5}^4,\zeta_5^4)$, $\diag(1,\zeta_{5},\zeta_5^4,1,1)$, $\diag(1,1,1,\zeta_{5},\zeta_5^4)$,  $(23)(45)$, $(24)(35)$, $(345)$ &$X_{1}$   &$48$  & $(29,5)$ \\
  91 & $\diag(1,\zeta_5,1,1,\zeta_{5}^4)$, $\diag(1,1,\zeta_5,1,\zeta_{5}^{4})$, $\diag(1,1,1,\zeta_5,\zeta_{5}^{4})$, $(12345)$, $(345)$ &$X_{1}$   &$24$  & $(15,3)$ \\

\end{longtable}

  }

  Table~\ref{table:quintic-landau-ginzburg-orbifold} uses the same notation as in
  Table~\ref{table:quintic-threefold-orbifold}.
  The second column shows generators of $G$ as a subgroup of $\GL_{5}(\C)$ and the third column shows an example
  of the nondegenerate quintic homogeneous polynomials $W$ whose maximal symmetry group contains $G$.

  {\small\renewcommand\arraystretch{1.5}
    \begin{longtable}{p{\dimexpr 0.06\linewidth-2\tabcolsep}
                  >{\raggedright\let\newline\\\arraybackslash\hspace{0pt}}p{\dimexpr 0.60\linewidth-2\tabcolsep}
                  p{\dimexpr 0.08\linewidth-2\tabcolsep}
                  p{\dimexpr 0.10\linewidth-2\tabcolsep}
                  p{\dimexpr 0.16\linewidth-2\tabcolsep}}
  \caption{The Hodge numbers and Witten indices of Landau-Ginzburg orbifold}
  \label{table:quintic-landau-ginzburg-orbifold}
  \\ \hline
  \# & Generators of $G$ (as a subgroup of $\GL_{5}(\C)$) & $W$ & $\tr(-1)^{F}$ & $(h^{1,1},\;h^{2,1})$ \\ \hline
  \endfirsthead
  \caption[]{\textit{(Continued from previous page)}} \\ \hline
  \# & Generators of $G$ (as a subgroup of $\GL_{5}(\C))$ & $W$  & $\tr(-1)^{F}$ &$(h^{1,1},\;h^{2,1})$ \\ \hline
  \endhead
  \hline \multicolumn{4}{r}{\textit{(Continued on next page)}} \\
  \endfoot
  \hline
  \endlastfoot

  1 & $J$ & Any & $200$ & $(101, 1)$ \\
  2 & $J$, $(12)(34)$ & $W_{1}$   &$112$  & $(59, 3)$ \\
  3 & $J$, $(123)$ & $W_{1}$   &$88$  & $(49, 5)$ \\
  4 & $J$, $\diag(\zeta_3,\zeta_3^2,\zeta_3,\zeta_3^2,1)$ &$W_{2}$   &$56$  & $(33, 5)$ \\
  5 & $J$, $\diag(\zeta_4,1,\zeta_4^3,1,1)$ &$W_{3}$   &$80$  & $(47, 7)$ \\
  6 & $J$, $\diag(\zeta_4,\zeta_4,-1,1,1)$ &$W_{4}$   &$32$  & $(27, 11)$ \\
  7 & $J$, $(12)(34)$, $(13)(24)$ &$W_{1}$   &$56$  & $(35, 7)$ \\
  8 & $J$, $\diag(1,\zeta_5,\zeta_5^4,1,1)$ &$W_{1}$   &$88$  & $(49, 5)$ \\
  9 & $J$, $\diag(\zeta_5,\zeta_5,\zeta_5^4,\zeta_5^4,1)$ &$W_{1}$   &$-8$  & $(17, 21)$ \\
  10 & $J$, $\diag(1,\zeta_5,\zeta_5^2,\zeta_5^3,\zeta_5^4)$ &$W_{1}$   &$40$  & $(21, 1)$ \\
  11 & $J$, $\diag(\zeta_3,\zeta_3^2,1,1,1)$, $(12)(34)$ &$W_{2}$   &$56$  & $(33, 5)$ \\
  12 & $J$, $\diag(\zeta_3,\zeta_3^2,\zeta_3,\zeta_3^2,1)$,$(12)(34)$ &$W_{2}$   &$40$  & $(25, 5)$ \\
  13 & $J$, $\diag(\zeta_3,\zeta_3^2,\zeta_3,\zeta_3^2,1)$,$(13)(24)$ &$W_{2}$   &$16$  & $(19, 11)$ \\
  14 & $J$, $\mathscr{A}_{1}$&$W_{5}$   &$-8$  & $(19, 23)$ \\
  15 & $J$, $\diag(\zeta_8,-1,1,\zeta_8^5,\zeta_4^3)$ &$W_{6}$   &$-8$  & $(13, 17)$ \\
  16 & $J$, $\diag(\zeta_4,\zeta_4,-1,1,1)$, $\diag(1,-1,-1,1,1)$ &$W_{4}$   &$4$  & $(21, 19)$ \\
  17 & $J$, $(12)(34)$, $\diag(1,1,\zeta_4,\zeta_4^3,1)$ &$W_{7}$   &$40$  & $(29, 9)$ \\
  18 & $J$, $\mathscr{A}_{2}$, $\diag(\zeta_4^3,\zeta_4,1,1,1)$ &$W_{8}$   &$64$  & $(41, 9)$ \\
  19 & $J$, $\diag(\zeta_3,\zeta_3^2,1,1,1)$, $\diag(1,1,\zeta_3,\zeta_3^2,1)$ &$W_{2}$   &$8$  & $(21, 17)$ \\
  20 & $J$, $(12)(34)$,$\diag(\zeta_5,\zeta_5^4,\zeta_5,\zeta_5^4,1)$ &$W_{1}$   &$8$  & $(17, 13)$ \\
  21 & $J$, $(12)(34)$, $\diag(\zeta_5,\zeta_5^4,1,1,1)$ &$W_{1}$   &$56$  & $(33, 5)$ \\
  22 & $J$, $(12)(34)$,$\diag(\zeta_5,\zeta_5^4,\zeta_5^2,\zeta_5^3,1)$ &$W_{1}$   &$32$  & $(19, 3)$ \\
  23 & $J$, $\diag(-1,1,-\zeta_5,\zeta_5,\zeta_5^3)$ &$W_{3}$   &$-16$  & $(11, 19)$ \\
  24 & $J$, $\diag(\zeta_3,\zeta_3^2,\zeta_3,\zeta_3^2,1)$, $\diag(1,1,\zeta_4,\zeta_4^3,1)$ &$W_{9}$   &$-16$  & $(15, 23)$ \\
  25 & $J$, $(12)(34)$, $(13)(24)$, $(123)$ &$W_{1}$   &$40$  & $(27, 7)$ \\
  26 & $J$, $\diag(\zeta_3,\zeta_3^2,\zeta_3,\zeta_3^2,1)$,$(13)(24)$,$(12)(34)$ &$W_{2}$   &$8$  & $(15, 11)$ \\
  27 & $J$, $\diag(\zeta_{13},\zeta_{13}^9,\zeta_{13}^3,1,1)$ &$W_{10}$   &$-88$  & $(5, 49)$ \\
  28 & $J$, $\diag(\zeta_3,\zeta_3^2,1,1,1)$, $\diag(1,1,\zeta_5,\zeta_5,\zeta_5^3)$ &$W_{2}$   &$-88$  & $(5, 49)$ \\
  29 & $J$, $\diag(\zeta_3,\zeta_3^2,1,1,1)$, $\diag(1,1,\zeta_5,1,\zeta_5^4)$ &$W_{11}$   &$8$  & $(21, 17)$ \\
  30 & $J$, $\diag(\zeta_3,\zeta_3^2,\zeta_3,\zeta_3^2,1)$,$\diag(1,1,\zeta_5,\zeta_5,\zeta_5^3)$ &$W_{2}$   &$-56$  & $(5, 33)$ \\
  31 & $J$, $\mathscr{A}_{1}$, $\mathscr{A}_{2}$ &$W_{5}$   &$8$  & $(21, 17)$ \\
  32 & $J$, $\diag(\zeta_{17},\zeta_{17}^{-4},\zeta_{17}^{16},\zeta_{17}^4,1)$ &$W_{12}$   &$-56$  & $(5, 33)$ \\
  33 & $J$, $\diag(\zeta_3,\zeta_3^2,1,1,1)$,$\diag(1,1,\zeta_3,\zeta_3^2,1)$, $(13)(24)$ &$W_{2}$   &$-8$  & $(13, 17)$ \\
  34 & $J$, $\diag(\zeta_3,\zeta_3^2,1,1,1)$,$\diag(1,1,\zeta_3,\zeta_3^2,1)$, $(12)(34)$ &$W_{2}$   &$16$  & $(19, 11)$ \\
  35 & $J$, $\diag(\zeta_4,1,\zeta_4^3,1,1)$, $\diag(1,1,\zeta_5,\zeta_5,\zeta_5^3)$ &$W_{3}$   &$-32$  & $(11, 27)$ \\
  36 & $J$, $\diag(-\zeta_5,\zeta_5,-\zeta_5^4,\zeta_5^4,1)$,$(13)(24)$ &$W_{3}$   &$-8$  & $(11, 15)$ \\
  37 & $J$, $\mathscr{A}_{2}$, $\mathscr{A}_{3}$&$W_{5}$   &$-0$  & $(13, 13)$ \\
  38 & $J$, $\diag(\zeta_3,\zeta_3^2,\zeta_3,\zeta_3^2,1)$,$\diag(1,1,\zeta_4,\zeta_4^3,1)$, $(12)(34)$ &$W_{7}$   &$-8$  & $(13, 17)$ \\
  39 & $J$, $\diag(\zeta_5,\zeta_5^4,1,1,1)$, $\diag(1,1,\zeta_5,\zeta_5^4,1)$ &$W_{1}$   &$8$  & $(21, 17)$ \\
  40 & $J$, $\diag(\zeta_5,\zeta_5^4,1,1,1)$, $\diag(1,1,\zeta_5,\zeta_5,\zeta_5^3)$ &$W_{1}$   &$-88$  & $(5, 49)$ \\
  41 & $J$, $\diag(1,\zeta_5,\zeta_5,\zeta_5^4,\zeta_5^4)$, $\diag(1,1,\zeta_5,\zeta_5,\zeta_5^3)$ &$W_{1}$   &$-40$  & $(1, 21)$ \\
  42 & $\diag(1,\zeta_5,\zeta_5^2,\zeta_5^3,\zeta_5^4)$,$(12345)$ &$W_{1}$   &$8$  & $(5, 1)$ \\
  43 & $J$, $\diag(\zeta_3,\zeta_3^2,\zeta_5,\zeta_5,\zeta_5^3)$,$(12)(34)$ &$W_{2}$   &$-56$  & $(5, 33)$ \\
  44 & $J$, $\diag(\zeta_3,\zeta_3^2,\zeta_{15}^8,\zeta_{15}^{13},\zeta_5^3)$,  $(12)(34)$ &$W_{2}$   &$-40$  & $(5, 25)$ \\
  45 & $J$, $\diag(\zeta_{15}^{11},\zeta_{15}^1,\zeta_{15}^{14},\zeta_{15}^{4},1)$, $(13)(24)$ &$W_{2}$   &$-40$  & $(5, 25)$ \\
  46 & $J$, $\diag(\zeta_3,\zeta_3^2,\zeta_5,\zeta_5^4,1)$,  $(12)(34)$ &$W_{11}$   &$16$  & $(19, 11)$ \\
  47 & $J$, $\diag(\zeta_{15}^{11},\zeta_{15}^1,\zeta_{15}^{14},\zeta_{15}^{4},1)$, $(14)(23)$ &$W_{2}$   &$-16$  & $(11, 19)$ \\
  48 & $J$, $\diag(\zeta_{17},\zeta_{17}^{-4},\zeta_{17}^{-1},\zeta_{17}^{4},1)$, $(13)(24)$ &$W_{12}$   &$-16$  & $(11, 19)$ \\
  49 & $J$, $\diag(\zeta_{3},\zeta_{3}^2,1,1,1)$,$\diag(1,1,\zeta_{3},\zeta_{3}^{2},1)$, $(14)(23)$,$(13)(24)$ &$W_{2}$   &$-16$  & $(9, 17)$ \\
  50 & $J$, $\diag(\zeta_{13},\zeta_{13}^9,\zeta_{13}^{3},1,1)$, $(123)$ &$W_{10}$   &$-8$  & $(17, 21)$ \\
  51 & $J$, $\diag(\zeta_{13},\zeta_{13}^9,\zeta_{13}^{3},1,1)$, $\mathscr{A}_{4}$ &$W_{10}$   &$-40$  & $(1, 21)$ \\
  52 & $J$, $\diag(\zeta_{13},\zeta_{13}^9,\zeta_{13}^{3},1,1)$, $\diag(1,1,1,\zeta_3,\zeta_3^2)$ &$W_{10}$   &$-200$  & $(1, 101)$ \\
  53 & $J$, $\diag(\zeta_{20}^9,\zeta_{5},\zeta_{20}^{11},\zeta_{5}^{4},1)$, $(13)(24)$ &$W_{3}$   &$-16$  & $(11, 19)$ \\
  54 & $J$, $\diag(\zeta_{41},\zeta_{41}^{-4},\zeta_{41}^{16},\zeta_{41}^{18},\zeta_{41}^{10})$ &$W_{13}$   &$-200$  & $(1, 101)$ \\
  55 & $J$, $\diag(\zeta_{3},\zeta_{3}^2,\zeta_5,\zeta_5,\zeta_5^3)$, $\diag(1,1,\zeta_{3},\zeta_{3}^2,1)$ &$W_{2}$   &$-200$  & $(1, 101)$ \\
  56 & $J$, $\mathscr{A}_{1}$, $\mathscr{A}_{2}$, $\mathscr{A}_{3}$ &$W_{5}$   &$-24$  & $(7, 19)$ \\
  57 & $J$, $\diag(1,1,\zeta_{5},\zeta_{5},\zeta_{5}^{3})$,$\diag(\zeta_{5},\zeta_{5}^4,\zeta_{5},\zeta_{5}^{4},1)$, $(12)(34)$ &$W_{1}$   &$-8$  & $(13, 17)$ \\
  58 & $J$, $\diag(1,1,\zeta_{5},\zeta_{5},\zeta_{5}^{3})$, $\diag(\zeta_{5},\zeta_{5}^4,1,1,1)$, $(12)(34)$ &$W_{1}$   &$-56$  & $(5, 33)$ \\
  59 & $J$, $\diag(1,1,\zeta_{5},\zeta_{5},\zeta_{5}^{3})$,$\diag(\zeta_{5},\zeta_{5}^4,\zeta_{5}^2,\zeta_{5}^{3},1)$, $(12)(34)$ &$W_{1}$   &$-32$  & $(3, 19)$ \\
  60 & $J$, $\diag(1,1,\zeta_{5},\zeta_{5}^4,1)$, $\diag(\zeta_{5},\zeta_{5}^4,1,1,1)$, $(12)(34)$ &$W_{1}$   &$16$  & $(19, 11)$ \\
  61 & $J$, $\diag(\zeta_{51}^{20},\zeta_{51}^{22},\zeta_{51}^{14},\zeta_{51}^{46},1)$ &$W_{12}$   &$-200$  & $(1, 101)$ \\
  62 & $J$, $(12345)$,  $(123)$ &$W_{1}$   &$16$  & $(13, 5)$ \\
  63 & $J$, $\diag(\zeta_{15}^8,\zeta_{15}^{13},\zeta_{15}^2,\zeta_{15}^{7},1)$, $(12)(34)$, $(13)(24)$ &$W_{2}$   &$-32$  & $(5, 21)$ \\
  64 & $J$, $\diag(1,1,1,\zeta_{5},\zeta_{5}^{4})$, $\diag(\zeta_{13},\zeta_{13}^{-4},\zeta_{13}^3,1,1)$ &$W_{14}$   &$-200$  & $(1, 101)$ \\
  65 & $J$, $\diag(1,1,\zeta_{5},1,\zeta_{5}^{4})$, $(345)$ &$W_{1}$   &$-8$  & $(17, 21)$ \\
  66 & $J$, $\diag(1,1,\zeta_{5},1,\zeta_{5}^{4})$, $\mathscr{A}_{5}$ &$W_{11}$   &$-40$  & $(1, 21)$ \\
  67 & $J$, $\diag(1,1,\zeta_{5},1,\zeta_{5}^{4})$, $\diag(\zeta_3,\zeta_3^2,1,\zeta_5,\zeta_5^4)$ &$W_{11}$   &$-200$  & $(1, 101)$ \\
  68 & $J$, $\diag(\zeta_{15}^{8},\zeta_{15}^{13},\zeta_{5}^4,\zeta_{5}^{4},1)$,$\diag(1,1,\zeta_{3},\zeta_{3}^2,1)$, $(14)(23)$&$W_{2}$   &$-112$  & $(3, 59)$ \\
  69 & $J$, $\diag(\zeta_3,\zeta_3^2,\zeta_{5},\zeta_5, \zeta_{5}^{3})$,$\diag(1,1,\zeta_{3},\zeta_{3}^2,1)$, $(12)(34)$ &$W_{2}$   &$-112$  & $(3, 59)$ \\
  70 & $J$, $\diag(\zeta_5,\zeta_5,\zeta_{5}^4,\zeta_{5}^{4},1)$,$\diag(\zeta_5,\zeta_5^4,\zeta_{5},\zeta_{5}^{4},1)$, $(12)(34)$, $(13)(24)$ &$W_{1}$   &$-16$  & $(9, 17)$ \\
  71 & $J$, $\diag(\zeta_3,\zeta_3^2,\zeta_{3},\zeta_{3}^{2},1)$, $\diag(\zeta_{17},\zeta_{17}^{-4},\zeta_{17}^{16},\zeta_{17}^{4},1)$, $(13)(24)$ &$W_{12}$   &$-112$  & $(3, 59)$ \\
  72 & $J$, $\diag(\zeta_{13},\zeta_{13}^{-4},\zeta_{13}^3,\zeta_{3},\zeta_3^2)$, $(123)$ &$W_{10}$   &$-88$  & $(5, 49)$ \\
  73 & $\diag(1, \zeta_5,\zeta_5^4,\zeta_{5}^4,\zeta_{5})$,   $(12345)$ &$W_{1}$   &$-8$  & $(1, 5)$ \\
  74 & $J$, $\diag(1,\zeta_5,1,1,\zeta_{5}^4)$, $\diag(1,1,\zeta_5,1,\zeta_{5}^{4})$, $\diag(1,1,1,\zeta_5,\zeta_{5}^{4})$ &$W_{1}$   &$-200$  & $(1, 101)$ \\
  75 & $J$, $\diag(1,1, \zeta_5,1,\zeta_{5}^4)$, $\diag(1,1,1,\zeta_5,\zeta_{5}^{4})$,  $(345)$, $(12)(34)$ &$W_{1}$   &$-16$  & $(11, 19)$ \\
  76 & $J$, $\diag(1,1, \zeta_5,1,\zeta_{5}^4)$, $\diag(1,1,1,\zeta_5,\zeta_{5}^{4})$,  $\mathscr{A}_{5}$, $(12)(34)$ &$W_{11}$   &$-32$  & $(3, 19)$ \\
  77 & $J$, $\diag(1,1, \zeta_5,\zeta_5,\zeta_{5}^3)$,$\diag(\zeta_3,\zeta_3^2,\zeta_5,\zeta_{5}^{4},1)$, $(12)(34)$ &$W_{11}$   &$-112$  & $(3, 59)$ \\
  78 & $J$, $\diag(1,1, \zeta_3,\zeta_{3}^2,1)$, $\diag(\zeta_{15}^8,\zeta_{15}^{13},\zeta_5^4,\zeta_{5}^{4},1)$,  $(12)(34)$, $(13)(24)$ &$W_{2}$   &$-80$  & $(1, 41)$ \\
  79 & $J$, $\diag(\zeta_{13},\zeta_{13}^{-4}, \zeta_{13}^3,\zeta_5,\zeta_{5}^4)$,  $(123)$ &$W_{14}$   &$-88$  & $(5, 49)$ \\
  80 & $J$, $\diag(\zeta_{41},\zeta_{41}^{-4},\zeta_{41}^{16},\zeta_{41}^{18},\zeta_{41}^{10})$, $(12345)$ &$W_{13}$   &$-40$  & $(1, 21)$ \\
  81 & $J$, $\diag(\zeta_3,\zeta_3^2, \zeta_5,1,\zeta_5^4)$, $\diag(1,1,1,\zeta_{5},\zeta_5^4)$,  $(345)$ &$W_{11}$   &$-88$  & $(5, 49)$ \\
  82 & $\diag(1,\zeta_5, \zeta_5^4,\zeta_{5}^4,\zeta_5)$, $\diag(1,\zeta_{5},\zeta_5^2,\zeta_{5}^{3},\zeta_5^4)$,  $(12345)$, $(25)(34)$ &$W_{1}$   &$-16$  & $(3, 11)$ \\
  83 & $J$, $\diag(1,\zeta_5, \zeta_5,\zeta_{5}^4,\zeta_5^4)$, $\diag(1,\zeta_{5},\zeta_5^4,1,1)$, $\diag(1,1,1,\zeta_{5},\zeta_5^4)$,  $(23)(45)$ &$W_{1}$   &$-112$  & $(3, 59)$ \\
  84 & $J$, $\diag(\zeta_5, \zeta_5^4,1,1,1)$, $\diag(1,1,\zeta_{5},1,\zeta_5^4)$, $\diag(1,1,1,\zeta_{5},\zeta_5^4)$,  $(345)$ &$W_{1}$   &$-88$  & $(5, 49)$ \\
  85 & $J$, $\diag(\zeta_3,\zeta_3^2,\zeta_5,1,\zeta_{5}^4)$, $\diag(1,1,1,\zeta_{5},\zeta_5^4)$, $(345)$,  $(12)(34)$ &$W_{11}$   &$-56$  & $(5, 33)$ \\
  86 & $J$, $\diag(1,\zeta_5, \zeta_5,\zeta_{5}^4,\zeta_5^4)$, $\diag(1,\zeta_{5},\zeta_5^4,1,1)$, $\diag(1,1,1,\zeta_{5},\zeta_5^4)$,  $(23)(45)$, $(24)(35)$ &$W_{1}$   &$-80$  & $(1, 41)$ \\
  87 & $\diag(1,\zeta_5,1,1,\zeta_{5}^4)$, $\diag(1,1,\zeta_5,1,\zeta_{5}^{4})$, $\diag(1,1,1,\zeta_5,\zeta_{5}^{4})$, $(12345)$ &$W_{1}$   &$-40$  & $(1, 21)$ \\
  88 & $J$, $\diag(\zeta_5, \zeta_5^4,1,1,1)$, $\diag(1,1,\zeta_{5},1,\zeta_5^4)$, $\diag(1,1,1,\zeta_{5},\zeta_5^4)$,  $(345)$, $(12)(34)$ &$W_{1}$   &$-56$  & $(5, 33)$ \\
  89 & $\diag(1,\zeta_5,1,1,\zeta_{5}^4)$, $\diag(1,1,\zeta_5,1,\zeta_{5}^{4})$, $\diag(1,1,1,\zeta_5,\zeta_{5}^{4})$, $(12345)$, $(25)(34)$ &$W_{1}$   &$-32$  & $(3, 19)$ \\
  90 & $J$, $\diag(1,\zeta_5, \zeta_5,\zeta_{5}^4,\zeta_5^4)$, $\diag(1,\zeta_{5},\zeta_5^4,1,1)$, $\diag(1,1,1,\zeta_{5},\zeta_5^4)$,  $(23)(45)$, $(24)(35)$, $(345)$ &$W_{1}$   &$-48$  & $(5, 29)$ \\
  91 & $\diag(1,\zeta_5,1,1,\zeta_{5}^4)$, $\diag(1,1,\zeta_5,1,\zeta_{5}^{4})$, $\diag(1,1,1,\zeta_5,\zeta_{5}^{4})$, $(12345)$, $(345)$ &$W_{1}$   &$-24$  & $(3, 15)$ \\

\end{longtable}

  }

\section{Calculation of Hodge numbers in the proof of Theorem~\ref{thm:main-theorem}}
  \label{sec:hodge-number-calculation}
  In the proof of Theorem~\ref{thm:main-theorem}, we calculated the Hodge numbers $h^{p, q}(W, G)$ using the formula
  in Proposition~\ref{thm:generalized-vafas-formula}.
  However, it is a hard and boring task to perform these computations by hand due to many calculations of eigenvalues,
  polynomial expansions, or derivation of conjugacy classes of large order groups.
  We should be helped by a computer algebra system such as Sage~\cite{sagemath}.
  We sketch the algorithm to compute the Poincaré polynomial using Sage in the following:
  \begin{enumerate}
    \item Construct $G$ from the generators.
      This can be accomplished by the Sage function \verb|MatrixGroup|.
      \verb|MatrixGroup| is a function which constructs the group with finitely many generators.
      For group operations used here such as \verb|MatrixGroup|, Sage internally uses GAP~\cite{GAP4}.
      GAP seems to fail (or, at least, take an impractically long time) to construct a matrix group if $G$ contains an element
      whose order is large to a certain extent.
      Fortunately, we can avoid this difficulty for the groups listed in
      Table~\ref{table:quintic-landau-ginzburg-orbifold} by constructing $G'$ generated by the generators of $G$ except $J$.
      Then, we can get representatives of conjugacy classes of $G$ using the fact that
      \begin{equation*}
        S = \set{J^{i}g \in G | 0 \leq i \leq 4, g \in S'}
      \end{equation*}
      where $S'$ is a set of representatives of the conjugacy classes of $G'$
      and the centralizer
      \begin{equation*}
        C_{G}(J^{i}g) = \set{J^{i}h \in G | 0 \leq i \leq 4, h \in C_{G'}(g)}
      \end{equation*}
      for each element in $S$.
      For example, the 32${}^{\text{nd}}$ and 61${}^{\text{st}}$ group in Table~\ref{table:quintic-landau-ginzburg-orbifold} may need
      such a roundabout way.
    \item Determine a set $S$ of representatives of conjugacy classes of $G$.
      We can do this by the method \verb|conjugacy_class_representatives|.
    \item Repeat the following steps for each $g$ in $S$.
      This step may fail for the 91${}^{\text{st}}$ group in Table~\ref{table:quintic-landau-ginzburg-orbifold}.
      This group $G$ has order $37500$ and $77$ conjugacy classes.
      The problem arises during the calculation for the conjugacy class corresponding to the identity matrix whose
      centralizer is the whole $G$.
      However, the Poincaré polynomial of this component can be calculated by hand.
      Indeed, we see
      \begin{equation*}
        \mathscr{H}_{W, G, E_{n}} \simeq \C {\cdot} 1 \oplus \C x_{1} x_{2} x_{3} x_{4} x_{5} \oplus
        \C x_{1}^{2} x_{2}^{2} x_{3}^{2} x_{4}^{2} x_{5}^{2} \oplus \C x_{1}^{3} x_{2}^{3} x_{3}^{3} x_{4}^{3} x_{5}^{3}
      \end{equation*}
      and
      \begin{equation*}
        P_{E_{n}}(W, G; u, v) = u^{3}v^{3} + u^{2}v^{2} + u v + 1.
      \end{equation*}
      \begin{enumerate}
        \item Compute the centralizer $C_{G}(g)$ and the fixed subspace $V^{g}$.
        \item Calculate the eigenvalues $\lambda_{i}$ of the restricted matrix of $h \in C_{G}(g)$ to $V^{g}$.
          Since we are considering the homogeneous case, all eigenvalues correspond to the charge $1/5$.
        \item Take the summation of
          \begin{equation*}
            \prod_{i = 1}^{n_{g}} \frac{\lambda_{i} - (u v)^{4/5}}{1 - \lambda_{i}(u v)^{1/5}}
          \end{equation*}
          over the elements of $C_{G}(g)$.
          However, Sage cannot handle the non-integral exponents.
          We should multiply $5$ to all exponents of $u$ and $v$ to avoid this difficulty.
          As a result, a monomial corresponding to $h^{p, q}(W, G)$ has the form $h^{p, q}(W, G) u^{5p} v^{5q}$,
          which means that our Poincaré polynomial here will be a polynomial like
          \begin{align*}
            & u^{15}v^{15} + h^{1,1}(W, G)u^{10}v^{10} \\
            & + u^{15} + h^{2,1}(W, G)u^{10}v^{5} + h^{2,1}(W, G)u^{5}v^{10} + v^{15} \\
            & + h^{1,1}(W, G)u^{5}v^{5} + 1.
          \end{align*}
        \item  Finally, return the above result multiplying
          \begin{equation*}
            u^{\age(g) - (5 - n_{g})/5}v^{\age(g^{-1})-(5 - n_{g})/5} / \# C_{G}(g).
          \end{equation*}
      \end{enumerate}
    \item Take the summation of all the results in the previous step, which gives the Poincaré polynomial $P(W, G; u, v)$.
  \end{enumerate}

\nocite{*}
\bibliographystyle{abbrv}
\bibliography{paper}

\end{document}